\documentclass[12pt]{article}
\usepackage{fontspec, xunicode, xltxtra}  
\usepackage[slantfont,boldfont]{xeCJK}
\usepackage{mathtools}
\usepackage{amsmath,diagbox}
\usepackage{indentfirst}
\usepackage{mathrsfs}
\usepackage{amsfonts}
\usepackage{arydshln}
\usepackage{enumerate}
\usepackage{setspace}
\usepackage{amssymb,amsthm,cases}
\usepackage{amsbsy,pifont}
\usepackage{latexsym,diagbox,tabularx}
\usepackage{amsmath,amscd,bbm,multirow}
\usepackage{graphicx,epsfig,extarrows,dsfont,mathtools}
\usepackage[final,allcolors=blue,colorlinks=true]{hyperref}
\usepackage[normalem]{ulem}
\usepackage{cite}

\input xy

\usepackage{caption,subcaption} 
\usepackage{tikz}
\usepackage{tikz-cd}
\usetikzlibrary{decorations.pathreplacing,decorations.markings}
\usetikzlibrary{cd}
\usetikzlibrary{shapes.geometric}
\usetikzlibrary{arrows,arrows.meta}
\usetikzlibrary{graphs,positioning}
\usetikzlibrary{matrix,decorations.pathmorphing}
\usetikzlibrary{math,calc,intersections,through,angles,arrows.meta,shapes.geometric,shadows,quotes,spy,datavisualization,datavisualization.formats.functions,plotmarks}
\tikzset{every picture/.style={samples=300,smooth,line join=round,thick,>=stealth}}
\usepackage{enumerate}
\usepackage{fancyhdr}
\usepackage[sectionbib]{chapterbib}

\tikzset{
	on each segment/.style={
		decorate,
		decoration={
			show path construction,
			moveto code={},
			lineto code={
				\path [#1]
				(\tikzinputsegmentfirst) -- (\tikzinputsegmentlast);
			},
			curveto code={
				\path [#1] (\tikzinputsegmentfirst)
				.. controls
				(\tikzinputsegmentsupporta) and (\tikzinputsegmentsupportb)
				..
				(\tikzinputsegmentlast);
			},
			closepath code={
				\path [#1]
				(\tikzinputsegmentfirst) -- (\tikzinputsegmentlast);
			},
		},
	},
	mid arrow/.style={postaction={decorate,decoration={
				markings,
				mark=at position .5 with {\arrow[#1]{stealth}}
	}}},
}

\numberwithin{equation}{section}

\theoremstyle{plain}
\newtheorem{theorem}{Theorem}[section]

\newtheorem{corollary}[theorem]{Corollary}

\newtheorem{lemma}[theorem]{Lemma}
\newtheorem{proposition}[theorem]{Proposition}

\theoremstyle{definition}

\newtheorem{remark}[theorem]{Remark}

\def\XXint#1#2#3{{\setbox0=\hbox{$#1{#2#3}{\int}$}
		\vcenter{\hbox{$#2#3$}}\kern-.5\wd0}}

\DeclareMathSymbol{\subseteq}{\mathrel}{symbols}{"12}
\DeclareMathSymbol{\supseteq}{\mathrel}{symbols}{"13} 
\DeclareMathSymbol{\subsetneq}{\mathrel}{AMSb}{"28}                  \DeclareMathSymbol{\supsetneq}{\mathrel}{AMSb}{"29}    
          
\DeclareMathSymbol{\nsubseteq}{\mathrel}{AMSb}{"2A}                  

\DeclareMathSymbol{\nsupseteq}{\mathrel}{AMSb}{"2B}

\DeclareMathDelimiter{\langle}{\mathop}{symbols}{"68}{largesymbols}{"0A}
\DeclareMathDelimiter{\rangle}{\mathclose}{symbols}{"69}{largesymbols}{"0B}

\DeclareSymbolFont{txfontsA}{U}{txmia}{m}{it}
\SetSymbolFont{txfontsA}{bold}{U}{txmia}{bx}{it}
\DeclareFontSubstitution{U}{txmia}{m}{it}
\DeclareMathSymbol{\upalpha}{\mathord}{txfontsA}{"0B}
\DeclareMathSymbol{\upbeta}{\mathord}{txfontsA}{"0C}
\DeclareMathSymbol{\upgamma}{\mathord}{txfontsA}{"0D}
\DeclareMathSymbol{\updelta}{\mathord}{txfontsA}{"0E}
\DeclareMathSymbol{\upepsilon}{\mathord}{txfontsA}{"0F}
\DeclareMathSymbol{\upzeta}{\mathord}{txfontsA}{"10}
\DeclareMathSymbol{\upeta}{\mathord}{txfontsA}{"11}
\DeclareMathSymbol{\uptheta}{\mathord}{txfontsA}{"12}
\DeclareMathSymbol{\upiota}{\mathord}{txfontsA}{"13}
\DeclareMathSymbol{\upkappa}{\mathord}{txfontsA}{"14}
\DeclareMathSymbol{\uplambda}{\mathord}{txfontsA}{"15}
\DeclareMathSymbol{\upmu}{\mathord}{txfontsA}{"16}
\DeclareMathSymbol{\upnu}{\mathord}{txfontsA}{"17}
\DeclareMathSymbol{\upxi}{\mathord}{txfontsA}{"18}
\DeclareMathSymbol{\uppi}{\mathord}{txfontsA}{"19}
\DeclareMathSymbol{\uprho}{\mathord}{txfontsA}{"1A}
\DeclareMathSymbol{\upsigma}{\mathord}{txfontsA}{"1B}
\DeclareMathSymbol{\uptau}{\mathord}{txfontsA}{"1C}
\DeclareMathSymbol{\upupsilon}{\mathord}{txfontsA}{"1D}
\DeclareMathSymbol{\upphi}{\mathord}{txfontsA}{"1E}
\DeclareMathSymbol{\upchi}{\mathord}{txfontsA}{"1F}
\DeclareMathSymbol{\uppsi}{\mathord}{txfontsA}{"20}
\DeclareMathSymbol{\upomega}{\mathord}{txfontsA}{"21}
\DeclareMathSymbol{\upvarepsilon}{\mathord}{txfontsA}{"22}
\DeclareMathSymbol{\upvartheta}{\mathord}{txfontsA}{"23}
\DeclareMathSymbol{\upvarpi}{\mathord}{txfontsA}{"24}
\DeclareMathSymbol{\upvarrho}{\mathord}{txfontsA}{"25}
\DeclareMathSymbol{\upvarsigma}{\mathord}{txfontsA}{"26}
\DeclareMathSymbol{\upvarphi}{\mathord}{txfontsA}{"27}                   

\DeclareSymbolFont{ugmL}{OMX}{mdugm}{m}{n}
\SetSymbolFont{ugmL}{bold}{OMX}{mdugm}{b}{n}
\DeclareMathAccent{\wideparen}{\mathord}{ugmL}{"F3}

\def\pt{\partial}
\def\mc{\mathcal}
\def\ra{\rightarrow}

\def\s{\subseteq}

\def\e{\epsilon}

\def\ol{\overline}
\def\ul{\underline}

\def\vp{\varphi}

\def\mf{\mathfrak}
\def\bf{\textbf}
\def\pt{\partial}

\def\om{\omega}
\def\Om{\Omega}
\def\la{\lambda}
\def\al{\alpha}
\def\be{\beta}
\def\de{\delta}

\def\Ga{\Gamma}

\def\ts{\times}

\def\iy{\infty}

\def\f{\frac}

\def\df{\mathrm d}

\def\hra{\hookrightarrow}

\def\diag{{\rm diag}}

\def\mcH{\mathcal{H}}

\def\mcL{\mathcal{L}}

\def\mcS{\mathcal{S}}

\def\mcA{\mathcal{A}}

	\def\Ee{\mathbbm{1}}

	\DeclareMathOperator{\Div}{div}

	\newcommand{\N}{\mathbb N}

\begin{document} 	
\title{Optimal damping design and exponential stabilization for weakly degenerate wave equations}
\author{Mengze Gu$^{a}$, Dong-Hui Yang$^b$\footnote{Corresponding author: donghyang@outlook.com}, Hongli Sun$^c$, Yuanzhi Zhou$^b$
	\\
	$^a${\it School of Mathematics and Science, Changsha Normal University}\\ {\it  Changsha 410083, China}\\
	$^b${\it School of Mathematics and Statistics, Central South University}\\
	{\it Changsha 410075, P.R.China}\\
	$^c${\it School of Mathematical and Physical Sciences}\\
	{\it Chongqing University of Science and Technology,	Chongqing 401331, China}\\  }
\date{}

\maketitle{}
\thispagestyle{empty}

\begin{abstract}
We investigate the optimal damping design and exponential stabilization for a class of higher‑dimensional weakly degenerate damped wave equations. In weighted Sobolev spaces, we establish well‑posedness and uniform exponential energy decay, and verify the existence of long‑time optimal damping potentials. To overcome analytical difficulties arising from system non-self-adjointness, we use finite‑time approximation via semigroup operator norms. Through linearization and adjoint‑state analysis, we derive first-order necessary optimality conditions and show finite‑time optimal dampings possess a bang‑bang structure.
\vspace{0.3cm}

\noindent {\bf {Keywords:}}  Degenerate Schr\"odinger equation; Interior degeneracy point; Weighted Sobolev space; Observability inequality

\vspace{0.3cm}

\noindent {\bf {AMS subject classifications (2020):}}~ 35J70, 35K65, 49Q10, 93B07.

\end{abstract}
	
\section{Introduction}

In this paper, we consider the following degenerate damped wave equation:
\begin{equation}\label{08.12.1}
	\begin{cases}
		\partial_{tt}\varphi-\Div(A\nabla\varphi)
		+V(x)\partial_t\varphi=0,
		&\mbox{in }\Omega\times(0,+\infty),\\
		\varphi=0,
		&\mbox{on }\partial\Omega\times(0,+\infty),\\
		\varphi(0)=\varphi^0,\quad
		\partial_t\varphi(0)=\varphi^1,
		&\mbox{in }\Omega.
	\end{cases}
\end{equation}
Here $\Omega\subset\overline{\mathbb R_+^N}$ is a bounded $C^2$ domain, and $
\Gamma_N^0=\partial\Omega\cap\{x_N=0\}$ 
is a nonempty relatively open subset of $\partial\Omega$. The diffusion matrix is given by,
\begin{equation*}
	A=\diag(\underbrace{1,\ldots,1}_{N-1},w),
	\quad
	w=x_N^\alpha,
	\quad
	\alpha\in(0,1).
\end{equation*}
Thus, the operator degenerates on the boundary portion $\Gamma_N^0$. The corresponding weighted Sobolev space $H_0^1(\Omega;w)$ will be introduced in Section~\ref{S2}. For $u\in H_0^1(\Omega;w)$, we define the degenerate elliptic operator,
\begin{equation*}
	\mathcal A u=-\Div(A\nabla u).
\end{equation*}
The initial data are assumed to satisfy,
$$
\varphi^0\in H_0^1(\Omega;w),
\quad
\varphi^1\in L^2(\Omega).
$$
The function $V$ represents the damping potential in \eqref{08.12.1}. Given constants
$0<H<K$ and $\beta\in(H,K)$, we introduce the admissible class,
\begin{equation}\label{08.24.2}
	\mathcal S(H,K;\beta)
	=
	\left\{
		W\in L^\infty(\Omega)\colon 
		H\leq W\leq K\ \mbox{a.e.~in }\Omega,\ 
		\int_\Omega W\mathrm dx=\beta|\Omega|
		\right\}.
\end{equation}
The constraint on the average of the damping potential represents a fixed total amount of damping, while the pointwise bounds prevent the damping from becoming arbitrarily large or vanishing.

For each $V\in\mathcal S(H,K;\beta)$, let
$$
\Phi(t)=
\begin{pmatrix}
	\varphi(t)\\
	\partial_t\varphi(t)
\end{pmatrix}
$$
denote the corresponding state. We are interested in the exponential decay rate of the associated semigroup. More precisely, for a fixed constant $C_0>1$, we define
\begin{equation}\label{08.24.6}
\omega(V) = \sup\left\{	\omega>0\colon 	\|\Phi(t)\|_{\mathcal H}	\leq	C_0e^{-\omega t}\|\Phi(0)\|_{\mathcal H}, 	\mbox{ for all }t\geq0	\mbox{ and all }\Phi(0)\in\mathcal H	\right\},
\end{equation}
where $\mathcal H=H_0^1(\Omega;w)\times L^2(\Omega)$.  The use of a fixed constant $C_0$ is important in the optimization problem, since it allows the decay estimate to be compared uniformly for different admissible damping potentials.

The main objective of this paper is to determine an optimal damping potential which maximizes the exponential decay rate. Namely, we consider the optimization problem
\begin{equation}\label{08.24.1}
	\sup_{V\in\mathcal S(H,K;\beta)}\omega(V).
\end{equation}

Stabilization of hyperbolic equations and its close relation to controllability have been extensively studied for uniformly hyperbolic systems \cite{Alabau,Chen,Chen1,Coron,Fragnelli,Komornik,Lasiecka,Lasiecka1,Lions,Rousseau,Russell,Triggiani,Weiss,Zuazua}. In contrast, existing results for degenerate hyperbolic equations are restricted to the one-dimensional setting  \cite{Alabau,Fragnelli,Zhang}. Extending these results to higher dimensions introduces additional technical difficulties \cite{Yang,Yang2,Yang1}, including the analysis of conormal derivatives, justification of integration-by-parts, and the establishment of adapted Poincar\'e and Hardy-type inequalities.

On the other hand, optimal damping design can be viewed as a shape optimization problem since it depends on the domain geometry and the structure of the solution \cite{Buttazzo,Chenais,Greco,Guo2,Guo1,Guo,He,Henrot, Tiba, Privat,Yang1}. However, optimizing the exponential decay rate \eqref{08.24.6} is considerably more delicate. In fact, the evolution operator associated with \eqref{08.12.1} is non-self-adjoint, and its spectrum alone does not determine the decay rate, rendering standard spectral optimization inapplicable.

This work provides the first optimal damping design for the exponential stabilization of higher-dimensional weakly degenerate wave equations. After establishing system well-posedness, we prove the existence of an optimal damping potential for \eqref{08.24.1} by balancing the weak-star topology of $L^\infty(\Omega)$ on the space of admissible potentials with the long-time behavior of the non-self-adjoint system. However, directly characterizing the optimal potential $V^*$ via $\omega(V)$ presents severe analytical obstacles (see Section \ref{S4}, Remarks \ref{08.28.R2} and \ref{08.28.R1}). To overcome this, we introduce an auxiliary finite-time optimization problem. Under a natural attainability assumption, Theorem \ref{08.29.T2} yields a bang-bang characterization, uncovering key structural properties of optimal damping distributions without requiring an explicit solution for \eqref{08.24.1}.

The main results of this paper are established in three principal theorems. Theorem \ref{08.25.T1} proves the existence of a long-time optimal damping potential. Theorems \ref{08.29.T1} and \ref{08.29.T2} establish key structural properties and a bang-bang characterization for the finite-time optimal damping. The paper is organized as follows. In Section \ref{S2}, we introduce the weighted energy space, analyze the spectral properties of $\mathcal A$, and establish system well-posedness and uniform exponential stability. Section \ref{S3} is devoted to proving the weak-star upper semicontinuity of the decay rate $\omega(V)$ to deduce the existence of an optimal potential. Finally, Section \ref{S4} introduces the auxiliary finite-time optimization problem and derives the optimality conditions leading to the bang-bang structure.

\section{Preliminaries}\label{S2}
In this section, we establish the preliminary results needed in the sequel. In particular, we introduce the energy space, study the spectral properties of the operator $\mathcal A$, establish the associated semigroup, and prove the exponential decay of solutions.

\subsection{Energy spaces}

Define
\begin{equation*}
	H^1(\Om;w)=\left\{u\in L^2(\Om)\colon \int_\Om \nabla u\cdot A\nabla u\df x<+\iy\right\}, 
\end{equation*}
where the derivatives are understood in the weak sense. 
Its inner product and norm are defined by 
\begin{equation*}
	(u,v)_{H^1(\Om;w)}=\int_\Om \left(uv+\nabla u\cdot A\nabla v\right)\df x, \quad \|u\|_{H^1(\Om;w)}=(u,u)_{H^1(\Om;w)}^\f{1}{2}. 
\end{equation*}
Set
\begin{equation*}
	H_0^1(\Om;w)=\mbox{the closure of $C_0^\iy(\Om)$ in } H^1(\Om;w). 
\end{equation*}
We denote by $H^{-1}(\Om;w)$  the dual space of $H_0^1(\Om;w)$ with $L^2(\Om)$ as the pivot space. 

Denote 
\begin{equation*}
	D(\mcA)=H_0^1(\Om;w)\cap \{u\colon \mcA u\in L^2(\Om)\}. 
\end{equation*}
Then 
\begin{equation*}
	(\mcA u, v)_{L^2(\Om)}=(u,v)_{H_0^1(\Om;w)}, \mbox{ for all } u\in D(\mcA) \mbox{ and all } v\in H_0^1(\Om;w). 
\end{equation*}

Throughout this paper, we denote 
\begin{equation}\label{08.23.1}
	x=(x',x_N),\quad M=\sup_{x\in\Om}|x|. 
\end{equation}

The following lemma is a Hardy's inequality. 

\begin{lemma}\label{08.23.L1}
	For each $u\in H_0^1(\Om;w)$, we have 
	\begin{equation*}
		\int_\Om x_N^{\al-2}u^2\df x\leq \f{4}{(1-\al)^2}\int_\Om x_N^\al (\pt_{x_N}u)^2\df x. 
	\end{equation*}
	
\end{lemma}

\begin{proof}
It is sufficient to prove the result for $u\in C_0^\iy(\Om)$, since this space is dense in $H_0^1(\Om;w)$. 
	
Indeed, for each $\be\in (\al,1)$, from 
\begin{equation*}
\int_s^M x_N^{\al-\be-1}\df x_N=\f{1}{\al-\be}\left(M^{\al-\be}-s^{\al-\be}\right)\leq \f{1}{\be-\al}s^{\al-\be}, 
\end{equation*}
we get 
\begin{equation*}
	\begin{split}
		\int_\Om x_N^{\al-2}u^2\df x
		&=\int_{(0,M)^{N-1}}\int_0^M x_N^{\al-2}|u(x',x_N)-u(x',0)|^2\df x_N\df x'\\
		&=\int_{(0,M)^{N-1}}\int_0^Mx_N^{\al-2}\left(\int_0^{x_N}\f{\pt u}{\pt x_N}(x',s)\df s\right)^2\df x_N\df x'\\
		&\leq \int_{(0,M)^{N-1}}\int_0^M x_N^{\al-2}\left(\int_0^{x_N} s^\be \left|\f{\pt u}{\pt x_N}(x',s)\right|^2\df s\right)\left(\int_0^{x_N}s^{-\be}\df s\right)\df x_N\df x'\\
		&=\f{1}{1-\be}\int_{(0,M)^{N-1}}\int_0^M\int_0^{x_N} x_N^{\al-\be-1}s^{\be}\left|\f{\pt u}{\pt x_N}(x',s)\right|^2\df s\df x_N\df x'\\	
		& =\f{1}{1-\be}\int_{(0,M)^{N-1}}\int_0^M\int_s^{M} x_N^{\al-\be-1}s^{\be}\left|\f{\pt u}{\pt x_N}(x',s)\right|^2\df x_N\df s\df x'\\
		&\leq \f{1}{(1-\be)(\be-\al)}\int_{(0,M)^{N-1}}\int_0^M s^\al \left|\f{\pt u}{\pt x_N}(x',s)\right|^2\df s\df x'. 
	\end{split}
\end{equation*}
Choosing $\be=\f{1+\al}{2}$, we obtain 
\begin{equation*}
	\int_\Om x_N^{\al-2}u^2\df x\leq \f{4}{(1-\al)^2}\int_\Om x_N^\al \left|\f{\pt u}{\pt x_N}\right|^2\df x. 
\end{equation*}
We complete the proof of this lemma. 
\end{proof}

\begin{remark}\label{08.23.R1}
	From Lemma \ref{08.23.L1}, we obtain that
	\begin{equation}\label{08.23.2}
		\int_\Omega x_N^{\alpha-2}u^2\df x
		\leq
		\frac{4}{(1-\alpha)^2}
		\int_\Omega \nabla u\cdot A\nabla u\df x .
	\end{equation}
	Moreover, since $0<x_N\leq M$ in $\Omega$, we have
	\begin{equation}\label{08.23.3}
		\int_\Omega u^2\df x
		=
		\int_\Omega x_N^{2-\alpha}x_N^{\alpha-2}u^2\df x
		\leq
		M^{2-\alpha}
		\int_\Omega x_N^{\alpha-2}u^2\df x
		\leq
		\frac{4M^{2-\alpha}}{(1-\alpha)^2}
		\int_\Omega \nabla u\cdot A\nabla u\df x .
	\end{equation}
	Hence, we obtain the Poincar\'e inequality. Denote
	\begin{equation}\label{08.24.3}
		C_P=\frac{4M^{2-\alpha}}{(1-\alpha)^2}. 
	\end{equation} Consequently,
	\begin{equation*}
		\|u\|_{H_0^1(\Omega;w)}
		:=
		\left(
		\int_\Omega \nabla u\cdot A\nabla u\df x
		\right)^{\f{1}{2}}
	\end{equation*}
	defines an equivalent norm on $H_0^1(\Omega;w)$, which will be used throughout this paper. Throughout this paper, we denote by
	\begin{equation*}
		a(u,v):=\int_\Omega \nabla u\cdot A\nabla v\,\mathrm dx, \mbox{ for all } 
		u,v\in H_0^1(\Omega;w),
	\end{equation*}
	the bilinear form associated with the degenerate elliptic operator
	$\mathcal A$.
\end{remark}

\begin{lemma}\label{08.23.L2}
	The embedding $H_0^1(\Om;w)\hra L^2(\Om)$ is compact. 
\end{lemma}

\begin{proof}
	Let $\{u_n\}_{n\in\N^*}\s H_0^1(\Om;w)$ be a bounded sequence. i.e., $\|u_n\|_{H_0^1(\Om;w)}^2\leq C_1$ for some positive constant $C_1$. Then there exists a subsequence of $\{u_n\}_{n\in\N^*}$, still denoted by itself, and $u_0\in H_0^1(\Om;w)$, such that 
	\begin{equation}\label{08.23.4}
		u_n\ra u_0 \mbox{ weakly in } H_0^1(\Om;w). 
	\end{equation}
	Replacing $u_n-u_0$ by $u_n$, we only need to prove the case $u_0=0$.

	Fix $\de>0$. Denote $\Om_\de=\Om\cap \{x_N>\de\}$, then, from \eqref{08.23.3}, we get 
	\begin{equation*}
		\int_{\Om-\Om_\de}u_n^2\df x=\int_{\Om-\Om_\de} x_N^{2-\al}x_N^{\al-2}u_n^2\df x\leq  \f{4\de^{2-\al}}{(1-\al)^2}\int_{\Om}\nabla u_n\cdot A\nabla u_n\df x\leq \f{4\de^{2-\al}C_1}{(1-\al)^2}. 
	\end{equation*}
	Let $\e>0$ be arbitrary. Choosing $\de_0>0$ small enough such that $\f{4\de_0^{2-\al}C_1}{(1-\al)^2}<\f{1}{4}\e^2$, then
	\begin{equation}\label{08.23.5}
		\|u_n|_{\Om-\Om_{\de_0}}\|_{L^2(\Om-\Om_{\de_0})}<\f{1}{2}\e. 
	\end{equation} 
	
	Now, since $H^1(\Om_{\de_0};w)= H^1(\Om_{\de_0})$ and $H^1(\Om_{\de_0})\hra L^2(\Om_{\de_0})$ is compact, from \eqref{08.23.4}, there exists a subsequence of $\{u_n|_{\Om_{\de_0}}\}_{n\in\N^*}$, still denoted by itself, such that 
	\begin{equation*}
		u_n|_{\Om_{\de_0}}\ra 0 \mbox{ strongly in } L^2(\Om_{\de_0})\mbox{ as } n\ra\iy. 
	\end{equation*}
	Choosing $n_0=n(\de_0,\e)\in\N^*$ such that for all $n \ge n_0$, $\|u_{n}\|_{L^2(\Omega_{\de_0})} < \frac 12 \epsilon$. Hence, for all $n \ge n_0$,
	\begin{equation*}
		\|u_{n}\|_{L^2(\Om)}\leq \|u_{n}|_{\Om-\Om_{\de_0}}\|_{L^2(\Om-\Om_{\de_0})}+\|u_{n}|_{\Om_{\de_0}}\|_{L^2(\Om_{\de_0})}<\e. 
	\end{equation*}
This completes the proof.
\end{proof}

\subsection{Spectrum}

For the partial differential operator 
\begin{equation}\label{08.25.11}
	\mcA: D(\mcA)\s L^2(\Om)\ra L^2(\Om), 
\end{equation}
from Lemma \ref{08.23.L2}, there exists a discrete point spectrum 
\begin{equation*}
	0<\la_1\leq \la_2\leq \la_3\leq \cdots, \quad \la_n\ra \iy \mbox{ as } n\ra\iy. 
\end{equation*}
Let the function $\phi_n$ denote the $n$th normalized eigenfunction of the operator $\mcA$ corresponding to the  $n$th eigenvalue $\la_n$. i.e., 
\begin{equation}\label{08.25.8}
	\begin{cases}
		\mcA \phi_n=\la_n\phi_n, &\mbox{in }\Om, \\
		\phi_n=0, &\mbox{on }\pt\Om. 
	\end{cases}
\end{equation}
Then $\{\phi_n\}_{n\in\N^*}$ is an orthonormal basis of $L^2(\Om)$. Note that from Remark \ref{08.23.R1}, we obtain 
\begin{equation*}
	\la_1\geq C_P^{-1}>0. 
\end{equation*}

Let $\theta\geq0$. Define
\begin{equation}\label{08.25.9}
D(\mcA^\theta)
=
\left\{
u\in L^2(\Om)\colon 
\sum_{n=1}^{\infty}
\la_n^{2\theta}
|(u,\phi_n)_{L^2(\Om)}|^2
<+\infty
\right\}.
\end{equation}
The inner product and the associated norm are defined by
\begin{equation*}
(u,v)_{D(\mcA^\theta)}
=
\sum_{n=1}^{\infty}
\la_n^{2\theta}
(u,\phi_n)_{L^2(\Om)}
(v,\phi_n)_{L^2(\Om)},
\end{equation*}
and
\begin{equation*}
\|u\|_{D(\mcA^\theta)}
=
\left(
\sum_{n=1}^{\infty}
\la_n^{2\theta}
|(u,\phi_n)_{L^2(\Om)}|^2
\right)^{\f{1}{2}}.
\end{equation*}
Then, by the spectral theorem for positive self-adjoint operators with
compact resolvent, 
$D(\mcA^\theta)$ endowed with the above inner product is a Hilbert space.
In particular,
$(D(\mcA^\theta),\|\cdot\|_{D(\mcA^\theta)})$
is a Banach space. Moreover, by the representation theorem for closed positive
quadratic forms, the form domain of $\mc A$ coincides with the
domain of its square root. Hence, \begin{equation}\label{08.25.10}
	D(\mcA^\f{1}{2})=H_0^1(\Om;w),  \mbox{ and } D(\mcA^1)=D(\mcA). 
\end{equation}
It is obvious that the embedding  
\begin{equation}\label{08.29.1}
	D(\mcA^s)\hra D(\mcA^\theta) \mbox{ is compact}
\end{equation}
for all $0\leq \theta <s$. 

\subsection{Semigroup and exponential decay}

Let 
\begin{equation}\label{08.25.6}
	\mathcal H=H_0^1(\Omega;w)\times L^2(\Omega).
\end{equation} 
We equip $\mathcal H$ with the inner product
\begin{equation*}
	\left\langle
	\begin{pmatrix}u\\v\end{pmatrix},
	\begin{pmatrix}z\\q\end{pmatrix}
	\right\rangle_{\mathcal H}
	=
	a(u,z)+(v,q)_{L^2(\Omega)}, \mbox{ with } a(u,z)=(u,z)_{H_0^1(\Om;w)}. 
\end{equation*}

We define
\begin{equation}\label{08.25.5}
	\mathfrak A=
	\begin{pmatrix}
		0&-I\\
		\mathcal A&V
	\end{pmatrix},
\end{equation}
with
\begin{equation}\label{08.25.7}
	\mathfrak A
	\begin{pmatrix}
		u\\
		v
	\end{pmatrix}
	=
	\begin{pmatrix}
		-v\\
		\mathcal A u+Vv
	\end{pmatrix},
	\qquad
	D(\mathfrak A)=D(\mathcal A)\times H_0^1(\Omega;w).
\end{equation}
Then the equation in \eqref{08.12.1} can be written as the abstract
evolution equation
\begin{equation}\label{08.23.6}
	\partial_t\Phi(t)+\mathfrak A\Phi(t)=0,  \mbox{ where } 
	\Phi(t)=
	\begin{pmatrix}
		\varphi(t)\\
		\partial_t\varphi(t)
	\end{pmatrix}. 
\end{equation}

For every $(u,v):=(u,v)^T\in D(\mathfrak A)$, we have
\begin{align*}
	\left\langle
	\mathfrak A
	\begin{pmatrix}u\\v\end{pmatrix},
	\begin{pmatrix}u\\v\end{pmatrix}
	\right\rangle_{\mathcal H}
	&=
	a(-v,u)+(\mathcal A u+Vv,v)_{L^2(\Omega)}
	\\
	&=
	-a(u,v)+a(u,v)+(Vv,v)_{L^2(\Omega)} =
	(Vv,v)_{L^2(\Omega)} \geq H\|v\|_{L^2(\Omega)}^2
	\geq0.
\end{align*}
Thus, $\mathfrak A$ is accretive.

We next prove that $\mathfrak A$ is maximal accretive. For every
$F=(f_1,f_2)^T\in\mathcal H$ and every $\lambda>0$, consider
\begin{equation}\label{08.23.7}
	(\lambda I+\mathfrak A)U=F,
\end{equation}
where $U=(u_1,u_2)^T\in D(\mathfrak A)$. Equivalently,
\begin{equation*}
	\begin{cases}
		\lambda u_1-u_2=f_1,\\
		\mathcal A u_1+Vu_2+\lambda u_2=f_2.
	\end{cases}
\end{equation*}
Hence $
u_2=\lambda u_1-f_1$, 
and therefore $
	\mathcal A u_1+\lambda(V+\lambda)u_1
	=
	f_2+(V+\lambda)f_1$. 
By the Lax--Milgram theorem, the variational problem
\begin{equation*}
	a(u_1,\psi)
	+
	\lambda\int_\Omega(V+\lambda)u_1\psi\, \df x
	=
	\int_\Omega
	\left[f_2+(V+\lambda)f_1\right]\psi\, \df x,
	\mbox{ for all } \psi\in H_0^1(\Omega;w),
\end{equation*}
admits a unique solution
$u_1\in H_0^1(\Omega;w)$. Moreover,
\[
\mathcal A u_1
=
f_2+(V+\lambda)f_1
-\lambda(V+\lambda)u_1
\in L^2(\Omega),
\]
and hence $u_1\in D(\mathcal A)$. Consequently,
$u_2=\lambda u_1-f_1\in H_0^1(\Omega;w)$, and
$U\in D(\mathfrak A)$.

Thus,
\begin{equation*}
R(\lambda I+\mathfrak A)=\mathcal H, \text{ for every }\lambda>0.
\end{equation*}
Since $\mathfrak A$ is accretive and
$R(\lambda I+\mathfrak A)=\mathcal H$ for some $\lambda>0$, 
$\mathfrak A$ is maximal accretive. In particular,
\begin{equation}\label{08.23.8}
	\mathfrak A:D(\mathfrak A)\s \mathcal H\ra\mathcal H \text{ is maximal accretive.}
\end{equation}

It follows from the Lumer-Phillips theorem that
$-\mathfrak A$ generates a $C_0$-semigroup of contractions on
$\mathcal H$. We denote this semigroup by
\begin{equation}\label{08.23.9}
	S_V(t)=e^{-t\mathfrak A},
	\mbox{ for all } t\geq0.
\end{equation}
Thus, for every $U_0\in\mathcal H$, the abstract Cauchy problem
\begin{equation*}
	\begin{cases}
		\partial_tU(t)+\mathfrak A U(t)=0,\\
		U(0)=U_0,
	\end{cases}
\end{equation*}
admits a unique mild solution $
U\in C([0,+\infty);\mathcal H)$, and $U(t)=S_V(t)U_0$. 
Moreover,
\begin{equation}\label{08.24.7}
	\|U(t)\|_{\mathcal H}
	\leq
	\|U_0\|_{\mathcal H},
	\mbox{ for all }  t\geq0.
\end{equation}

If $U_0\in D(\mathfrak A)$, then
\begin{equation}\label{08.30.2}
U\in C([0,+\infty);D(\mathfrak A))
\cap C^1([0,+\infty);\mathcal H),
\end{equation} 
and (see \cite[Theorem 7.4, p.~185]{Brezis})
\begin{equation}\label{08.31.1}
	\partial_tU(t)=-\mathfrak A U(t), \mbox{ and } \left\|\f{\df}{\df t}U(t)\right\|_\mcH = \|\mf{A}U(t)\|_{\mcH}\leq \|\mf{A}U_0\|_{\mcH} \mbox{ for all }t\geq 0.
\end{equation}

For the inhomogeneous equation (see Chapter IV, \cite{Pazy})
\begin{equation}\label{08.30.1}
	\begin{cases}
		\partial_tU(t)+\mathfrak A U(t)=F(t),
		& t\in[0,T],\\
		U(0)=U_0,
	\end{cases}
\end{equation}
the mild solution is given by the variation-of-constants formula
\begin{equation*}
	U(t)
	=
	S_V(t)U_0
	+
	\int_0^tS_V(t-s)F(s)\, \df s.
\end{equation*}
In particular, if $U_0\in D(\mathfrak A)$ and
$F\in C([0,T];\mathcal H)$, then
\begin{equation*}
U\in C([0,T];D(\mathfrak A))
\cap C^1([0,T];\mathcal H).
\end{equation*}

\begin{remark}\label{08.24.R1}
	Let $V_1,V_2\in\mc S(H,K;\beta)$. 
	Since $V_1,V_2\in L^\infty(\Om)$ and $
	H_0^1(\Om;w)\hookrightarrow L^2(\Om)$, 
	we have $
	(V_1-V_2)v\in L^2(\Om)$
	for every $v\in H_0^1(\Om;w)$.
	Hence, $
	\mc Au+V_1v\in L^2(\Om)$ 
	if and only if $
	\mc Au+V_2v\in L^2(\Om)$. 
	Consequently, the domain of $\mathfrak A_V$ is independent of
	the choice of $V\in\mc S(H,K;\beta)$, namely,
	\begin{equation*}
		D(\mathfrak A_{V_1})
		=
		D(\mathfrak A_{V_2})=D(\mcA)\ts H_0^1(\Om;w).
	\end{equation*}
\end{remark}

The following proposition establishes the exponential stability of the mild solution $\Phi(t)$ to \eqref{08.23.6}. In particular, the exponential decay rate is uniform with respect to the admissible damping potentials $V$. Note that $\|\Phi(t)\|_{\mcH}^2=2E(t)$.

\begin{proposition}\label{08.23.P1}
	Assume that $V\in L^\infty(\Omega)$ satisfies $
	V(x)\geq H>0$ in $\Omega$. 
	Then the semigroup generated by $-\mathfrak A$ is exponentially stable.
	More precisely, there exist constants $C>0$ and $\omega>0$ such that
	\begin{equation*}
	\|\Phi(t)\|_{\mathcal H}
	\leq
	Ce^{-\omega t}\|\Phi(0)\|_{\mathcal H},
	\mbox{ for all }  t\geq0,
	\end{equation*} where $\Phi$ is the mild solution of \eqref{08.23.6} with initial datum $
	\Phi(0)=\Phi_0\in\mathcal H$.
\end{proposition}

\begin{proof}
	It is sufficient to prove the estimate for
	$\Phi_0\in D(\mathfrak A)$.
	The general case follows by density $D(\mf{A})\s \mcH$  and the continuity of the
	$C_0$-semigroup.
	
	Let $\Phi_0\in D(\mf{A})$. 
	Define the energy
	\begin{equation}\label{08.25.4}
	E(t)
	=\f{1}{2}\|\Phi(t)\|_{\mcH}^2=
	\frac12
	\left(
	\int_\Omega\nabla\varphi\cdot A\nabla\varphi\df x
	+
	\int_\Omega|\partial_t\varphi|^2\df x
	\right).
	\end{equation}
	Multiplying the equation \eqref{08.12.1} (see   \eqref{08.23.6}) by $\partial_t\varphi\in H_0^1(\Om;w)$   and integrating over
	$\Omega$, from 
	\begin{equation*}
		(\mcA u,v)_{L^2(\Om)}=a(u,v), \mbox{ for all } u \in D(\mathcal A),v\in H_0^1(\Om;w), 
	\end{equation*} we obtain
	\begin{equation*}
	E'(t)
	=
	-\int_\Omega V|\partial_t\varphi|^2\df x .
	\end{equation*}
	Hence,
	\begin{equation}\label{08.23.11}
	E'(t)
	\leq
	-H\|\partial_t\varphi\|_{L^2(\Omega)}^2 .
	\end{equation}
	However, this estimate does not directly imply exponential decay since
	the energy also contains the term $
	\int_\Omega\nabla\varphi\cdot A\nabla\varphi\df x$. 
	
	To overcome this difficulty, we introduce the auxiliary functional
	\begin{equation*}
	F(t)
	=
	\int_\Omega\varphi\partial_t\varphi\df x .
	\end{equation*}
	Then $
	F'(t)
	=
	\int_\Omega|\partial_t\varphi|^2\df x
	+
	\int_\Omega\varphi\partial_{tt}\varphi\df x$. 
	Therefore, from $\vp\in H_0^1(\Om;w)$, we get 
	\begin{equation*}
	\begin{aligned}
		F'(t)
		&=
		\|\partial_t\varphi\|_{L^2(\Omega)}^2
		-\int_\Omega\nabla\varphi\cdot A\nabla\varphi\df x 
		-\int_\Omega V\varphi\partial_t\varphi\df x .
	\end{aligned}
	\end{equation*}
	Since $V\in L^\infty(\Omega)$ and by the weighted Poincar\'e inequality \eqref{08.23.3},
	we obtain, for any $\e>0$,
	\begin{equation*}
	\begin{aligned}
		\left|
		\int_\Omega V\varphi\partial_t\varphi\df x
		\right| 
		&\leq
		\|V\|_{L^\infty(\Omega)}
		\|\varphi\|_{L^2(\Omega)}
		\|\partial_t\varphi\|_{L^2(\Omega)}\\ &\leq
		\f{1}{4}K\e^{-1}
		\|\partial_t\varphi\|_{L^2(\Omega)}^2
		+C_PK\e
		\int_\Omega\nabla\varphi\cdot A\nabla\varphi\df x.
	\end{aligned}
	\end{equation*} 
	Consequently,
	\begin{equation*}
	F'(t)
	\leq
	\left(1+\f{1}{4}K\e^{-1}\right) 
	\|\partial_t\varphi\|_{L^2(\Omega)}^2
	-
	\left(1-C_PK\e\right)
	\int_\Omega\nabla\varphi\cdot A\nabla\varphi\df x. 
	\end{equation*} 
	
	Now define the Lyapunov functional
	\begin{equation*}
	L(t)=E(t)+\delta F(t),
	\end{equation*}
	where $\delta>0$ will be chosen later. Since  
	\begin{equation*} 
		\begin{split} 
	|F(t)|
	&\leq
	\|\varphi\|_{L^2(\Omega)}
	\|\partial_t\varphi\|_{L^2(\Omega)}
	\leq \f{1}{2\sqrt{C_P}}\|\vp\|_{L^2(\Om)}^2+\f{\sqrt{C_P}}{2}\|\pt_t\vp\|_{L^2(\Om)}^2\leq 
	\sqrt{C_P}E(t), 
	\end{split} 
	\end{equation*}  
	where the constant $C_P$ is defined in \eqref{08.24.3}. 
	For $0<\de<\f{1}{\sqrt{C_P}}$,  we have the equivalence
	\begin{equation*} 
	\left(1-\de\sqrt{C_P}\right) E(t)\leq L(t)\leq \left(1+\de\sqrt{C_P}\right) E(t).
	\end{equation*}  
	Moreover,
	\begin{equation*}
	\begin{aligned}
		L'(t)
		&=
		E'(t)+\delta F'(t)\\
		&\leq
		\left[-H +\de \left(1+\f{1}{4}K\e^{-1}\right) \right]
		\|\partial_t\varphi\|_{L^2(\Omega)}^2
		-
		\de \left(1-C_PK\e\right)
		\int_\Omega\nabla\varphi\cdot A\nabla\varphi\df x.
	\end{aligned}
	\end{equation*}
	Choosing   
	\begin{equation}\label{08.24.5}
		\e=\f{1}{2C_PK},  \mbox{  and } \de=\min\left\{\f{H}{2+C_PK^2}, \f{1}{2\sqrt{C_P}}\right\},
	\end{equation}
	then, 
	\begin{equation*}
		0<\de<\f{1}{\sqrt{C_P}},\quad H -\de \left(1+\f{1}{4}K\e^{-1}\right)>0,\mbox{ and }  \de \left(1-C_PK\e\right)>0. 
	\end{equation*}
	hence 
	\begin{equation*}
	L'(t) \leq -\f{H}{2}\|\pt_t\vp\|_{L^2(\Om)}^2-\f{\de}{2}\|\vp\|_{H_0^1(\Om;w)}^2 \leq -\min\{H,\de\}E(t). 
	\end{equation*}
	Using the equivalence between $L(t)$ and $E(t)$, we obtain
	\begin{equation*}
	L'(t)\leq-\omega_0 L(t), \mbox{ with } \om_0=\f{\min\{H,\de\}}{1+\de\sqrt{C_P}}. 
	\end{equation*} 
	By Gronwall's  inequality, $
	L(t)\leq L(0)e^{-\omega_0 t}$. 
	Finally, again using the equivalence of $L(t)$ and $E(t)$, $
	E(t)\leq 4e^{-\omega_0 t}E(0)$. 
	Which gives
	\begin{equation}\label{08.24.4}
	\|\Phi(t)\|_{\mathcal H}
	\leq
	2e^{-\omega t}\|\Phi(0)\|_{\mathcal H}, \mbox{ with } \om=\f{\om_0}{2}=\f{1}{2}\f{\min\{H,\de\}}{1+\de\sqrt{C_P}}.
	\end{equation} 
This completes the proof.
\end{proof}

\section{Shape design}\label{S3} 

In this section, we will  consider the problem \eqref{08.24.1}. The use of a fixed constant $C_0$  
is essential for the stability of the optimization problem, since it provides a uniform exponential estimate with respect to the admissible potentials. 

\begin{remark}
	\label{08.25.R1}
	The fixed constant $C_0$ in \eqref{08.24.6} is introduced to
	ensure the uniformity of the exponential decay estimate with respect
	to variations of the potential $V$. Indeed, when a sequence
	$V_n\ra V$ weak-star in $L^\infty(\Omega)$ is considered
	(see Theorem \ref{08.24.T1}), the uniform constant $C_0$ allows us to
	pass to the limit in the exponential estimate and obtain the upper
	semicontinuity of the decay rate.
	
	The particular choice of a sufficiently large $C_0$ only affects the
	multiplicative factor in the decay estimate and does not affect the
	asymptotic exponential decay rate. Indeed, if for some $C_\omega>0$
	and $\omega>0$,
	\begin{equation*}
	\|\Phi(t)\|_{\mathcal H}
	\leq
	C_\omega e^{-\omega t}
	\|\Phi(0)\|_{\mathcal H},
	\mbox{ for all }  t\geq0 \mbox{ and all } \Phi(0)\in\mcH,
	\end{equation*}
	then, for any $0<\omega'<\omega$,
	\begin{equation*}
	\|\Phi(t)\|_{\mathcal H}
	\leq
	C_\omega e^{-(\omega-\omega')t}
	e^{-\omega't}
	\|\Phi(0)\|_{\mathcal H}
	\leq
	C_\omega e^{-\omega't}
	\|\Phi(0)\|_{\mathcal H}, 
	\mbox{ for all }  t\geq0 \mbox{ and all } \Phi(0)\in\mcH.
	\end{equation*}
	Therefore, the multiplicative constant only affects the prefactor of
	the estimate, while the asymptotic decay exponent is characterized by
	the supremum of all admissible exponential rates. 
	More precisely, the multiplicative constant does not affect the
	supremal exponential rate, since for every $C_\omega>0$ and every
	$0<\omega'<\omega$, one can absorb the factor $C_\omega$ into the
	exponential term on a bounded time interval, while the contribution
	of the multiplicative constant to
	\begin{equation*}
	-\frac{1}{T}\log\bigl(C_\omega\|S_V(T)\|_{\mathcal L(\mathcal H)}\bigr)
	\end{equation*}
	is of order $T^{-1}$ and therefore vanishes as $T\to+\infty$.
\end{remark}

The following lemma shows that $\omega(V)$, $V\in\mathcal S(H,K;\beta)$, is bounded from below.

\begin{lemma}\label{08.24.L1}
	There exists a constant $\ul\omega>0$, depending only on
	$\al, H$, $K$ and  $\Omega$, such that
	\begin{equation*}
	\omega(V)\geq \ul\omega, 
	\mbox{ for all } V\in\mcS(H,K;\be).
	\end{equation*}
	In particular,
	\begin{equation*}
	\inf_{V\in\mcS(H,K;\be)}\omega(V)>0.
	\end{equation*}
\end{lemma}

\begin{proof}
	By \eqref{08.24.4}, for every
	$V\in\mathcal S(H,K;\beta)$, the exponential decay rate satisfies
	\begin{equation*}
		\omega(V)
		\geq
		\frac12
		\frac{\min\{H,\delta\}}
		{1+\delta\sqrt{C_P}},
	\end{equation*}
	where $\om(V)$ is defined in \eqref{08.24.6}, and  $\delta$ is defined in \eqref{08.24.5}, and
	$C_P$ is defined in \eqref{08.24.3}.
	
	Since every
	$V\in\mathcal S(H,K;\beta)$ satisfies $
	H\leq V\leq K$  a.e.~in $\Omega$, 
	the constants appearing in the proof of Proposition
	\ref{08.23.P1}, namely $\e$ and $\delta$, are independent of
	$V$.
	
	Therefore, setting
	\begin{equation*}
		\underline{\omega}
		=
		\frac12
		\frac{\min\{H,\delta\}}
		{1+\delta\sqrt{C_P}},
	\end{equation*}
	we have
	\begin{equation*}
		\omega(V)\geq\underline{\omega}>0, 
		\mbox{ for all }  V\in\mathcal S(H,K;\beta).
	\end{equation*}
	
	Consequently,
	\begin{equation*}
		\inf_{V\in\mathcal S(H,K;\beta)}
		\omega(V)
		\geq
		\underline{\omega}>0 .
	\end{equation*}
	This proves the lemma.
\end{proof}

\begin{lemma}\label{08.24.L2}
	There exists a constant $\overline\omega=K$ such that
	\begin{equation*}
		\omega(V)\leq\overline\omega,
		\mbox{ for all } V\in\mcS(H,K;\beta).
	\end{equation*}
\end{lemma}

\begin{proof}
	It is sufficient to consider strong solutions with initial data
	$\Phi(0)\in D(\mathfrak A)$, since the corresponding estimate extends
	to $\mathcal H$ by density. 
	
	Let  $V\in\mcS(H,K;\beta)$ be arbitrary. From \eqref{08.23.11}, we obtain 
	\begin{equation*}
		E'(t)
		=
		-\int_\Omega V|\partial_t\varphi|^2\df x .
	\end{equation*}
	Since $V\leq K$ in $\Omega$, we have
	\begin{equation*}
		E'(t)
		\geq
		-K\|\partial_t\varphi(t)\|_{L^2(\Omega)}^2
		\geq
		-2K E(t).
	\end{equation*}
	Consequently, by Gronwall's inequality,
	\begin{equation*}
		E(t)\geq e^{-2Kt}E(0),
		\mbox{ for all } t\geq0 \mbox{ and all }\Phi(0)\in D(\mf A).
	\end{equation*}
	
	Now let $0<\omega<\omega(V)$. By the definition of $\omega(V)$,  
	\begin{equation*}
		\|\Phi(t)\|_{\mathcal H}
		\leq
		C_0 e^{-\omega t}\|\Phi(0)\|_{\mathcal H},
		\mbox{ for all } t\geq0 \mbox{ and for all } \Phi(0)\in D(\mf A).
	\end{equation*}
	Equivalently, $
		E(t)\leq C_0^2e^{-2\omega t}E(0)$. 
	Combining the above two inequalities, we obtain
	\begin{equation*}
		e^{-2Kt}
		\leq
		C_0^2e^{-2\omega t},
		\mbox{ for all } t\geq0 \mbox{ and  all } \Phi(0)\in D(\mf A).
	\end{equation*}
	Hence, $
		e^{2(\omega-K)t}\leq C_0^2 $. 
	Letting $t\to+\infty$ yields $
		\omega\leq K$. 
	Since this holds for every $\omega<\omega(V)$, we conclude that
	\begin{equation*}
		\omega(V)\leq K .
	\end{equation*}
	Therefore,
	\begin{equation*}
		\omega(V)\leq\overline\omega:=K,
		\mbox{ for all }  V\in\mcS(H,K;\beta).
	\end{equation*}
	This proves the lemma. 
\end{proof}

Lemma \ref{08.24.L2} shows that $\omega(V)$, $V\in\mathcal S(H,K;\beta)$, is bounded from above.
Now, we establish the following convergence result for the solutions
$\varphi_n\ (n\in\N^*)$ corresponding to the damping potentials $V_n$.

\begin{lemma}\label{08.24.L4}
	Let $\{V_n\}_{n\in\mathbb N}\subset\mc S(H,K;\beta)$
	satisfy
	\begin{equation*}
		V_n\ra V \mbox{ weak star in }L^\infty(\Omega).
	\end{equation*}
	Assume that the initial data satisfy
	\begin{equation*}
		(\varphi^0,\varphi^1)\in D(\mathfrak A).
	\end{equation*}
	Let $\varphi_n$ and $\varphi$ be the corresponding strong
	solutions of \eqref{08.12.1} associated with $V_n$ and $V$,
	respectively. Then, for every $t\in[0,T]$,
	\begin{equation*}
		E_{V_n}\ra E_V \mbox{ uniformly in } [0,T],
	\end{equation*}
	where
	\begin{equation*}
		E_{V}(t)
		=
		\frac12
		\left(
		\|\partial_t\varphi(t)\|_{L^2(\Omega)}^2
		+
		a(\varphi(t),\varphi(t))
		\right), 
	\end{equation*}
	and  $a(\cdot,\cdot)$ is defined in Remark \ref{08.23.R1}. 
	Moreover,
	\begin{equation*}
		\varphi_n\ra \varphi \mbox{ strongly in }
		C([0,T];H_0^1(\Omega;w)), \mbox{ and } 
		\partial_t\varphi_n\ra\partial_t\varphi \mbox{ strongly in }
		C([0,T];L^2(\Omega)).
	\end{equation*}
\end{lemma}

\begin{proof}
	We prove this theorem by the following steps. 
	
	{\it Step 1}. 
	Since $
	(\varphi^0,\varphi^1)\in D(\mathfrak A)$ 
	and the domain of $\mathfrak A_{V}$ is independent of
	$V\in\mc S(H,K;\beta)$ (see Remark \ref{08.24.R1}), we have $
	\Phi_n(t)\in D(\mathfrak A),  t\in[0,T]$. 
	Hence,
	\begin{equation*}
	\partial_t\varphi_n(t)\in H_0^1(\Omega;w),
	\mbox{ for all }  t\in[0,T].
	\end{equation*}
	Moreover, by \eqref{08.24.7} and \eqref{08.31.1}, we obtain
	\begin{equation}\label{08.24.9}
		\begin{split} 
		\sup_{n\in\mathbb N^*}
		\sup_{t\in [0,T]}E_{V_n}(t)
		&\leq \f{1}{2}\left(\| \vp^0\|_{H_0^1(\Om;w)}^2+\|\vp^1\|_{L^2(\Om)}^2\right)<+\infty,\\ 
		\sup_{n\in\mathbb N^*}\sup_{t\in [0,T]} \|\pt_t\vp_n(t)\|_{H_0^1(\Om;w)}
		&\leq C\left(\|\vp^0\|_{D(\mcA)}+\|\vp^1\|_{H_0^1(\Om;w)}\right)<+\infty,
		\end{split} 
	\end{equation} 
	and 
	\begin{equation}\label{08.24.10}
	\sup_{n\in\mathbb N^*}
	\|\partial_{tt}\varphi_n\|_
	{L^\infty(0,T;L^2(\Omega))}\leq 
	C\left(\| \vp^0\|_{D(\mcA)}+\|\vp^1\|_{H_0^1(\Om;w)}\right)<+\infty,
	\end{equation}
	where $D(\mcA)$ is defined in \eqref{08.25.10},  and  the positive constant $C$ depends only on $K$.

Combining   \eqref{08.24.9} and \eqref{08.24.10}, there exist a subsequence, still denoted by
$\{\partial_t\varphi_n\}_{n\in\N^*}$, and a function $z$ such that
\begin{equation}\label{08.31.2}
	\begin{split}
		\vp_n
		&\ra z \mbox{ weak star  in }L^\iy(0,T; H_0^1(\Om;w)), \\ 
		\pt_t\vp_n 
		&\ra \pt_tz \mbox{ weak star  in } L^\iy(0,T; H_0^1(\Om;w)),\\
		\pt_{tt}\vp_n
		&\ra \pt_{tt}z \mbox{ weakly in } L^2(Q).
	\end{split}
\end{equation}
By Lemma \ref{08.23.L2} and the Aubin-Lions-Simon compactness theorem, the sequence $\{\partial_t\varphi_n\}_{n\in\N^*}$ is relatively compact in $C([0,T];L^2(\Omega))$. Hence, there exists a subsequence, still denoted by $\{\partial_t\varphi_n\}_{n\in\N^*}$,  such that
\begin{equation}\label{08.31.3}
	\begin{split}  
	\partial_t\varphi_n
	&\to \pt_tz \mbox{ strongly in }
	C([0,T];L^2(\Omega)). 
	\end{split} 
\end{equation}

	{\it Step 2}. 
	It remains to identify the limit $z$. For any
	$\psi\in C_c^\infty(\Omega)$, using the weak formulation of the
	equation \eqref{08.12.1}, we have
	\begin{align*}
		&(\partial_{tt}\varphi_n,\psi)_{L^2(\Omega)}
		+a(\varphi_n,\psi)
		+(V_n\partial_t\varphi_n,\psi)_{L^2(\Omega)}
		=0 .
	\end{align*} 
	Passing to the limit, from \eqref{08.31.2},  the first two terms converge by the weak
	convergence of $\varphi_n$ and $\partial_t\varphi_n$, while from \eqref{08.31.3} and $V_n\ra V$ weak-star in $L^\iy(\Om)$ we obtain 
	\begin{equation*}
	V_n\partial_t\varphi_n
	\ra
	V\partial_tz \mbox{ weakly in }L^2(\Omega).
	\end{equation*} 
	Consequently, $z$ satisfies the limit equation with the same
	initial data. By the uniqueness of the strong solution, we obtain $
	z=\varphi$. 
	Therefore,
	\begin{equation}\label{08.24.8}
	\partial_t\varphi_n
	\to
	\partial_t\varphi \mbox{ strongly in }
	C([0,T];L^2(\Omega)).
	\end{equation}
	
	{\it Step 3}. 
	Finally, using the energy identity and
	$\partial_t\varphi_n\in H_0^1(\Omega;w)$, we have
	\begin{equation}\label{08.31.6}
		E_{V_n}(t)
		+
		\int_0^t\int_\Omega
		V_n|\partial_t\varphi_n|^2\,\mathrm{d}x\mathrm{d}s
		=E(0).
	\end{equation}
	We first show that
	\begin{equation*}
		\int_0^t\int_\Omega
		V_n|\partial_t\varphi_n|^2\,\mathrm{d}x\mathrm{d}s
		\ra 
		\int_0^t\int_\Omega
		V|\partial_t\varphi|^2\,\mathrm{d}x\mathrm{d}s
	\end{equation*}
	uniformly for $t\in[0,T]$.
	
	Indeed, from \eqref{08.24.9} and \eqref{08.24.8}, for every $t\in[0,T]$,
	\begin{align*}
		&\left|
		\int_0^t\int_\Omega
		V_n|\partial_t\varphi_n|^2\,\mathrm{d}x\mathrm{d}s
		-
		\int_0^t\int_\Omega
		V_n|\partial_t\varphi|^2\,\mathrm{d}x\mathrm{d}s
		\right| \\
		&\leq
		K
		\int_0^t\int_\Omega
		\left|
		|\partial_t\varphi_n|^2
		-
		|\partial_t\varphi|^2
		\right|
		\,\mathrm{d}x\mathrm{d}s \\
		&\leq
		K
		\|\partial_t\varphi_n-\partial_t\varphi\|_{L^2(Q)}
		\left(
		\|\partial_t\varphi_n\|_{L^2(Q)}
		+
		\|\partial_t\varphi\|_{L^2(Q)}
		\right)\\
		&\leq C\left(\|\vp^0\|_{D(\mcA)}+\|\vp^1\|_{H_0^1(\Om;w)}\right)\|\partial_t\varphi_n-\partial_t\varphi\|_{L^2(Q)},
	\end{align*}
	where the positive constant $C$ depends only on $K$. 
	Hence,
	\begin{equation}\label{08.31.4}
		\sup_{t\in[0,T]}
		\left|
		\int_0^t\int_\Omega
		V_n
		\left(
		|\partial_t\varphi_n|^2
		-
		|\partial_t\varphi|^2
		\right)
		\,\mathrm{d}x\mathrm{d}s
		\right|
		\ra 0.
	\end{equation}
	
	It remains to consider
	\begin{equation*}
		F_n(t)
		:=
		\int_0^t\int_\Omega
		(V_n-V)|\partial_t\varphi|^2
		\,\mathrm{d}x\mathrm{d}s.
	\end{equation*}
For each fixed $t\in[0,T]$, since $\int_0^t|\partial_t\varphi(s)|^2\,\mathrm{d}s
	\in L^1(\Omega)$, 
	$V_n\ra V$ weak-star  in $L^\infty(\Omega)$ implies $
		F_n(t)\ra 0$. 
	Moreover, the family $\{F_n\}_{n\in\mathbb N}$ is equicontinuous on
	$[0,T]$. Indeed, for $0\leq t_1<t_2\leq T$,
	\begin{align*}
		|F_n(t_2)-F_n(t_1)|
		&\leq
		\|V_n-V\|_{L^\infty(\Omega)}
		\int_{t_1}^{t_2}
		\|\partial_t\varphi(s)\|_{L^2(\Omega)}^2
		\,\mathrm{d}s \leq
		2K
		\int_{t_1}^{t_2}
		\|\partial_t\varphi(s)\|_{L^2(\Omega)}^2
		\,\mathrm{d}s.
	\end{align*}
	Since $
	\partial_t\varphi\in C([0,T];L^2(\Omega))$, 
	the function $
	s\mapsto
	\|\partial_t\varphi(s)\|_{L^2(\Omega)}^2$ 
	is continuous on $[0,T]$ and hence uniformly integrable. Therefore,
	$\{F_n\}_{n\in\mathbb N}$ is equicontinuous on $[0,T]$.
	By the Arzel\`a--Ascoli theorem, the family
	$\{F_n\}_{n\in\mathbb N}$ is relatively compact in
	$C[0,T]$. Since $F_n(t)\to0$ for every $t\in[0,T]$, every
	uniformly convergent subsequence has the limit $0$. Consequently,
	\begin{equation}\label{08.31.5}
		\sup_{t\in[0,T]}|F_n(t)|
		\ra 0.
	\end{equation}
	Combining \eqref{08.31.4} and \eqref{08.31.5}, we obtain
	\begin{equation*}
		\sup_{t\in[0,T]}
		\left|
		\int_0^t\int_\Omega
		V_n|\partial_t\varphi_n|^2
		\,\mathrm{d}x\mathrm{d}s
		-
		\int_0^t\int_\Omega
		V|\partial_t\varphi|^2
		\,\mathrm{d}x\mathrm{d}s
		\right|
		\ra 0.
	\end{equation*}
	Consequently, by the energy identities \eqref{08.31.6},
	\begin{equation*}
		E_{V_n}(t)\ra E_V(t), \mbox{ uniformly for }t\in[0,T].
	\end{equation*}
	Moreover, from
	\eqref{08.24.8} and the definition of the energy, we have
	\begin{equation*}
		a(\varphi_n(t),\varphi_n(t))
		\to
		a(\varphi(t),\varphi(t)),
		\mbox{ uniformly in }[0,T].
	\end{equation*}
	Since $
	\varphi_n(t)\ra \varphi(t)$ weakly in $H_0^1(\Omega;w)$ 
	for every $t\in[0,T]$, and $
	\|\varphi_n(t)\|_{H_0^1(\Omega;w)}
	\to
	\|\varphi(t)\|_{H_0^1(\Omega;w)}$, 
	the uniform convexity of
	$H_0^1(\Omega;w)$ yields
	\begin{equation}\label{08.31.7}
		\varphi_n(t)\to\varphi(t)
		\mbox{ strongly in }H_0^1(\Omega;w), \mbox{ for every } t\in [0,T]. 
	\end{equation} 
	
	From \eqref{08.24.9},  since
	\begin{equation*}
		\sup_{n\in\mathbb N^*}
		\|\partial_t\varphi_n\|_
		{L^\infty(0,T;H_0^1(\Omega;w))}
		<+\infty,
	\end{equation*}
	the sequence $\{\varphi_n\}_{n\in\mathbb N}$ is equicontinuous in
	$C([0,T];H_0^1(\Omega;w))$. Indeed, for
	$0\leq t_1<t_2\leq T$, we have
	\begin{align*}
		\|\varphi_n(t_2)-\varphi_n(t_1)\|_{H_0^1(\Omega;w)}^2
		&=
		\left\|
		\int_{t_1}^{t_2}
		\partial_t\varphi_n(s)\,\mathrm{d}s
		\right\|_{H_0^1(\Omega;w)}^2 \leq
		|t_2-t_1|
		\int_{t_1}^{t_2}
		\|\partial_t\varphi_n(s)\|_{H_0^1(\Omega;w)}^2
		\,\mathrm{d}s
		\\
		&\leq
		\left(
		\sup_{n\in\mathbb N}
		\|\partial_t\varphi_n\|_
		{L^\infty(0,T;H_0^1(\Omega;w))}
		\right)^2
		|t_2-t_1|^2.
	\end{align*}
Consequently, from \eqref{08.24.9}, 
\begin{equation*}
	\|\varphi_n(t_2)-\varphi_n(t_1)\|_{H_0^1(\Omega;w)}
	\leq C|t_2-t_1|,
\end{equation*}
where $C>0$ is independent of $n$, $t_1$ and $t_2$. Hence, $\{\varphi_n\}_{n\in\mathbb N}$ is equicontinuous in $C([0,T];H_0^1(\Omega;w))$. Combining this with \eqref{08.31.7} and  the Arzel\`a-Ascoli theorem yields
	\begin{equation*}
		\varphi_n\ra\varphi
		 \mbox{ strongly in }
		C([0,T];H_0^1(\Omega;w)).
	\end{equation*} 
	 We complete the proof of this lemma. 
\end{proof}

\begin{corollary}\label{08.29.C1}
Let $\{V_n\}_{n\in\mathbb N^*}\s\mathcal S(H,K;\beta)$ satisfy
	\begin{equation*}
		V_n\ra V \mbox{ weak star in }L^\infty(\Omega).
	\end{equation*}
	Then, for every
	\begin{equation*}
		\Phi_0 =(\varphi^0,\varphi^1)\in\mathcal H,
	\end{equation*}
	the corresponding weak solutions satisfy
	\begin{equation*}
		S_{V_n}(t)\Phi_0 
		\to
		S_V(t)\Phi_0  \mbox{ strongly in }\mathcal H,
	\end{equation*}
	uniformly for $t\in[0,T]$.
\end{corollary}

\begin{proof}
	Let $\Phi_0 \in\mathcal H$. Since
	$D(\mathfrak A)$ is dense in $\mathcal H$, there exists a sequence
	$\{\Phi_m^0\}_{m\in\mathbb N^*}\subset D(\mathfrak A)$ such that
	\begin{equation*}
		\Phi_m^0\to\Phi_0  \mbox{ strongly in }\mathcal H.
	\end{equation*}
	
	By \eqref{08.24.4} and Lemmas \ref{08.24.L1} and \ref{08.24.L2},   we get $
		\sup_{0\leq t\leq T}
		\|S_{V_n}(t)\|_{\mathcal L(\mathcal H)}
		\leq 2$, 
	and the same estimate holds for $S_V(t)$.
	
	For fixed $m\in\N^*$, Lemma \ref{08.24.L4} gives
	\begin{equation*}
		S_{V_n}(t)\Phi_m^0
		\to
		S_V(t)\Phi_m^0 \mbox{ strongly in }\mathcal H, \mbox{ uniformly for } t\in[0,T].
	\end{equation*}

	Moreover,
	\begin{align*}
		&\|S_{V_n}(t)\Phi_0 -S_V(t)\Phi_0 \|_{\mathcal H}\\
		&\leq
		\|S_{V_n}(t)(\Phi_0 -\Phi_m^0)\|_{\mathcal H}+
		\|S_{V_n}(t)\Phi_m^0-S_V(t)\Phi_m^0\|_{\mathcal H}+
		\|S_V(t)(\Phi_m^0-\Phi_0 )\|_{\mathcal H}.
	\end{align*}
	Taking the supremum over $t\in[0,T]$ and then the limit
	$n\to+\infty$, from above, we obtain
	\begin{equation*}
		\limsup_{n\to+\infty}
		\sup_{0\leq t\leq T}
		\|S_{V_n}(t)\Phi_0 -S_V(t)\Phi_0 \|_{\mathcal H}
		\leq
		4\|\Phi_0 -\Phi_m^0\|_{\mathcal H}.
	\end{equation*}
	Letting $m\to+\infty$ proves the result.
\end{proof}

\begin{theorem}\label{08.24.T1}
Let $\{V_n\}_{n\in\mathbb N}\s \mc S(H,K;\beta)$ satisfy
\begin{equation*}
	V_n\ra V \mbox{ weak star  in }L^\infty(\Omega).
\end{equation*}
Then,
\begin{equation*}
\limsup_{n\to+\infty}\omega(V_n) \leq	\omega(V).
\end{equation*}
\end{theorem}

\begin{proof}
We argue by contradiction. Assume that
\begin{equation*}
	\limsup_{n\to+\infty}\omega(V_n)	>		\omega(V).
\end{equation*}
Then there exists a constant $\omega$ such that
\begin{equation*}
\omega(V)<\omega <		\limsup_{n\to+\infty}\omega(V_n).
\end{equation*}
Hence, up to a subsequence, we have
	\begin{equation*}
		\omega(V_n)>\omega, \mbox{ for all } n\in\mathbb N^*.
	\end{equation*} 
	By the definition of the exponential decay rate \eqref{08.24.6},   
	\begin{equation}\label{08.24.20}
		E_{V_n}(t)
		\leq
		C_0^2 e^{-2\omega t}E_{V_n}(0)
		=
		C_0^2 e^{-2\omega t}E(0),
		\mbox{ for all } t\geq0 \mbox{ and all } \Phi(0)\in\mcH.
	\end{equation}
	 
	Now, by Corollary \ref{08.29.C1}, the convergence
	\begin{equation*}
		E_{V_n}(t)\ra E_V(t)
	\end{equation*}
	is uniform on every bounded interval $[0,T]$.
	Passing to the limit in \eqref{08.24.20}, we obtain
	\begin{equation*}
		E_V(t)
		\leq
		C_0^2 e^{-2\omega t}E(0), \mbox{ for all }t\geq 0 \mbox{ and all }\Phi(0)\in\mcH.
	\end{equation*} 
	Consequently, $\omega$ is an admissible exponential decay
	rate for the system associated with $V$. Hence, by the definition
	of $\omega(V)$, we have
	\begin{equation*}
		\omega\leq\omega(V).
	\end{equation*} 
	This contradicts the choice of $\omega$, namely
	\begin{equation*}
		\omega>\omega(V).
	\end{equation*} 
	Therefore,
	\begin{equation*}
		\limsup_{n\to+\infty}\omega(V_n)
		\leq
		\omega(V),
	\end{equation*}
	which proves the upper semicontinuity of the decay rate.
\end{proof}

We now proceed to prove one of the main results of this paper.

\begin{theorem}[Existence of an optimal potential]\label{08.25.T1}
There exists $V^*\in\mathcal S(H,K;\beta)$ such that
\begin{equation*}
\omega(V^*) = \sup_{V\in\mathcal S(H,K;\beta)} \omega(V).
\end{equation*}
\end{theorem}

\begin{proof}
Let  
\begin{equation}\label{08.25.3}
m:= \sup_{V\in\mathcal S(H,K;\beta)} \omega(V).
\end{equation}
Then, from Lemma \ref{08.24.L2}, we get $m\leq \ol\om$.  By the definition of the supremum, there exists a maximizing sequence $\{V_n\}_{n\in\mathbb N^*}\s\mathcal S(H,K;\beta)$ such that
\begin{equation}\label{08.25.1}
\omega(V_n)\rightarrow m \mbox{ as }n\rightarrow+\infty .
\end{equation}
	
Since $H\leq V_n\leq K$ a.e.~in $\Omega$,  the sequence $\{V_n\}_{n\in\N^*}$ is bounded in $L^\infty(\Omega)$. Hence, by the Banach-Alaoglu theorem, there exists a subsequence,
still denoted by $\{V_n\}_{n\in\N^*}$, and some $V^*\in L^\infty(\Omega)$ such that
\begin{equation}\label{08.25.2}
	V_n\ra V^* \mbox{ weak star in }L^\infty(\Omega).
\end{equation}

We next prove that $V^*\in\mathcal S(H,K;\beta)$. 
	
Indeed, for every nonnegative $\psi\in L^1(\Omega)$,
we have $\int_\Omega (V_n-H)\psi\df x\geq0$.  Passing to the limit yields $\int_\Omega (V^*-H)\psi\df x\geq0$,  and therefore $V^*\geq H$ a.e.~in $\Omega$.   Similarly, since $K-V_n\geq0$ a.e.~in $\Omega$, we obtain $K-V^*\geq0$ a.e.~in $\Omega$.  Hence,  $H\leq V^*\leq K$ a.e.~in $\Omega$.  Moreover, by the weak-star convergence \eqref{08.25.2},
\begin{equation*}
\int_\Om V^*\df x=\int_\Omega V^*\Ee_\Om \df x	=	\lim_{n\rightarrow+\infty}	\int_\Omega V_n\Ee_\Om \df x	=	\beta|\Omega|,
\end{equation*}
where $\Ee_\Om$ is the characteristic function of $\Om$.  Therefore, $V^*\in\mathcal S(H,K;\beta)$.

Finally, using Theorem \ref{08.24.T1}, we have
\begin{equation*}
\limsup_{n\rightarrow+\infty}\omega(V_n)	\leq \omega(V^*).
\end{equation*}
Since $\{V_n\}_{n\in\N^*}$ is a maximizing sequence, from \eqref{08.25.1}, we obtain  
\begin{equation*}
m\leq \omega(V^*).
\end{equation*}
On the other hand, by the definition of $m$ (see  \eqref{08.25.3}), $\omega(V^*)\leq m$. Therefore,
\begin{equation*}
\omega(V^*)=m .
\end{equation*}
Hence $V^*$ is an optimal potential. We complete the proof of this theorem. 
\end{proof}

The following corollary is a direct consequence of the Banach-Alaoglu theorem and the weak-star closedness of the admissible set.

\begin{corollary}\label{08.29.C3}
	The admissible set $\mathcal S(H,K;\beta)$ is weak-star compact in
	$L^\infty(\Omega)$.
\end{corollary}

\section{Characterize the optimal potential}\label{S4} 

In this section, we denote 
\begin{equation*}
	\mf{A}_V=
	\begin{pmatrix}
		0 & -I\\
		\mcA & V
	\end{pmatrix}.
\end{equation*}

\begin{lemma}[Equivalence of decay rates]
	\label{08.28.L1}
	Let $S_V(t)=e^{-t\mathfrak A_V}$  be the $C_0$-semigroup generated by
	$-\mathfrak A_V$. Let $\om(V)$ be defined in \eqref{08.24.6}, and  let
	\begin{equation*}
		\omega_0(V)
		:=
		-\lim_{T\to+\infty}
		\frac1T
		\log
		\|S_V(T)\|_{\mathcal L(\mathcal H)} .
	\end{equation*}
	Then
	\begin{equation*}
		\omega(V)=\omega_0(V).
	\end{equation*}
\end{lemma}

\begin{proof}
	The semigroup $S_V(t)$ generated by $-\mathfrak A_V$
	is defined in \eqref{08.23.9}. 
	
	We first prove that
	\begin{equation*}
	\omega(V)\leq\omega_0(V).
	\end{equation*}
	
	Let $\omega<\omega(V)$. By the definition of $\omega(V)$, we have $
	\|S_V(t)\Phi(0)\|_{\mathcal H}
	\leq
	C_0e^{-\omega t}\|\Phi(0)\|_\mcH
	\mbox{ for all } t\geq0$ and all $\Phi(0)\in \mcH$. Then 
	\begin{equation*}
		\|S_V(t)\|_{\mcL(\mcH)}\leq C_0e^{-\om t}, \mbox{ for all }t\geq 0. 
	\end{equation*}
	Hence,
	\begin{equation*}
	\frac1t
	\log
	\|S_V(t)\|_{\mathcal L(\mathcal H)}
	\leq
	-\omega+\frac{\log C_0}{t}, \mbox{ for all }t\geq 0.
	\end{equation*}
	Passing to the limit as $t\to+\infty$, we obtain $
	-\omega_0(V)\leq-\omega$, 
	and therefore $
	\omega\leq\omega_0(V)$. 
	Taking the supremum over all $\omega<\omega(V)$ yields
	\begin{equation}\label{08.28.1}
	\omega(V)\leq\omega_0(V).
	\end{equation}
	
	Conversely, let $
	0<\omega<\omega_0(V)$. 
	By the definition of $\omega_0(V)$,
	\begin{equation*}
	\lim_{t\to+\infty}
	\frac1t
	\log
	\|S_V(t)\|_{\mathcal L(\mathcal H)}
	=
	-\omega_0(V).
	\end{equation*}
	Hence,  there exists $T_0>0$ such that for all $t\geq T_0$, $
	\|S_V(t)\|_{\mathcal L(\mathcal H)}
	\leq
	e^{-\omega t}$. 
	On the bounded interval $[0,T_0]$, the strong continuity of the
	semigroup implies
	\begin{equation*}
	M:=\sup_{0\leq t\leq T_0}
	e^{\omega t}
	\|S_V(t)\|_{\mathcal L(\mathcal H)}
	<+\infty .
	\end{equation*}
	Therefore,
	\begin{equation*}
	\|S_V(t)\|_{\mathcal L(\mathcal H)}
	\leq
	M e^{-\omega t},
	\mbox{ for all }  t\geq0 .
	\end{equation*}
	Since $\|S_V(t)\|_{\mcL(\mcH)}=\sup_{0\neq \Phi(0)\in \mcH}\f{\|S_V(t)\Phi(0)\|_{\mcH}}{\|\Phi(0)\|_{\mcH}}$, we obtain 
	\begin{equation*}
		\|S_V(t)\Phi(0)\|_\mcH\leq Me^{-\om t}\|\Phi(0)\|_\mcH, \mbox{ for all } t\geq 0 \mbox{ and all } \Phi(0)\in \mcH. 
	\end{equation*} The multiplicative constant does not affect the supremal
	exponential decay rate (see Remark \ref{08.25.R1}), we conclude that $
	\omega(V)\geq\omega$. 
	Letting $\omega\uparrow\omega_0(V)$ gives
	\begin{equation}\label{08.28.2}
	\omega(V)\geq\omega_0(V).
	\end{equation}
	
	Combining the two inequalities \eqref{08.28.1} and \eqref{08.28.2}, we obtain
	\begin{equation*}
	\omega(V)=\omega_0(V).
	\end{equation*}
	We prove the lemma. 
\end{proof}

\begin{remark}\label{08.28.R2}
An alternative approach to the exponential stability problem is to consider the quadratic operator pencil
	\begin{equation*}
		Q_V(\lambda)
		=
		\lambda^2 I+\lambda V+\mathcal A,
		\quad
		D(Q_V(\lambda))=D(\mathcal A).
	\end{equation*}
This provides a natural spectral formulation of the damped wave equation. However, using this approach for the present optimization problem would require a substantial amount of additional spectral analysis. In particular, since $\mathcal A$ is an unbounded operator, $Q_V(\lambda)$ is an unbounded operator pencil, and one has to study its resolvent, root vectors, and, most importantly, the completeness of the system of root vectors in $L^2(\Omega)$. Such issues are considerably more delicate for non-self-adjoint unbounded operator pencils than for self-adjoint elliptic operators.
	
In contrast, the semigroup formulation developed above allows us to obtain uniform energy estimates directly from the dissipativity of the damping term. Moreover, the finite-time convergence of the solutions under the weak-star convergence $V_n\ra V$  in $L^\iy(\Om)$ can be established by compactness arguments. This is particularly well suited to the optimization problem over $\mathcal S(H,K;\beta)$, since it leads directly to the upper semicontinuity of the exponential decay rate without requiring a complete spectral decomposition of the quadratic operator pencil. 
\end{remark}

\begin{remark}\label{08.28.R1}
The representation in Lemma \ref{08.28.L1} provides an alternative description of the exponential decay rate in terms of the asymptotic behavior of the semigroup. Nevertheless, directly studying the dependence of $\omega(V)$ on the damping potential $V$ through the quantity
\begin{equation*}
	-\lim_{T\to+\infty}
	\frac{1}{T}\log\|S_V(T)\|_{\mathcal L(\mathcal H)}
\end{equation*}
is difficult, since it involves the long-time behavior of the semigroup and does not directly interact with the weak-star topology of $L^\infty(\Omega)$.

For this reason, we adopt an approximation and compactness approach. The corresponding solutions can be shown to converge strongly on	every finite time interval with respect to the weak-star convergence $V_n\ra  V$ in $L^\iy(\Om)$. In particular, the associated energies converge uniformly on bounded time intervals. This finite-time convergence allows us to transfer a uniform exponential estimate from the approximating systems to the limiting system, which yields the upper semicontinuity of the decay rate established in Theorem \ref{08.24.T1}. 
\end{remark}

\begin{corollary} 
	\label{08.28.C1}
	Let $V\in\mathcal S(H,K;\beta)$ and let
	$S_V(t)=e^{-t\mathfrak A_V}$ be the $C_0$-semigroup generated by
	$-\mathfrak A_V$. Define
	\begin{equation}\label{08.28.3}
		\omega_T(V)
		:=
		\frac1T
		\log
		\frac{C_0}
		{\|S_V(T)\|_{\mathcal L(\mathcal H)}},
		\mbox{ for }  T>0 .
	\end{equation}
	Then
	\begin{equation*}
		\lim_{T\to+\infty}\omega_T(V)=\omega(V).
	\end{equation*}
\end{corollary}

\begin{proof}
	By the definition of $\omega_T(V)$,
	\begin{equation*}
		\omega_T(V)
		=
		-\frac1T
		\log
		\|S_V(T)\|_{\mathcal L(\mathcal H)}
		+
		\frac{\log C_0}{T}.
	\end{equation*}
	By Lemma \ref{08.28.L1},
	\begin{equation*}
		-\lim_{T\to+\infty}
		\frac1T
		\log
		\|S_V(T)\|_{\mathcal L(\mathcal H)}
		=
		\omega(V).
	\end{equation*}
	Since $
		\lim_{T\to+\infty}
		\frac{\log C_0}{T}=0$, 
	we conclude that
	\begin{equation*}
		\lim_{T\to+\infty}\omega_T(V)=\omega(V).
	\end{equation*}
	This proves the corollary. 
\end{proof}

\begin{lemma} 
	\label{08.28.L2}
	Let $V\in\mathcal S(H,K;\beta)$ and let
	$\omega_T(V)$ be defined in \eqref{08.28.3}.
	Define the finite-time terminal energy functional
	\begin{equation}\label{08.29.2}
		J_T(V)
		:=
		-\frac12
		\|S_V(T)\|_{\mathcal L(\mathcal H)}^2 .
	\end{equation}
	Then
	\begin{equation*}
		\omega_T(V)
		=
		\frac1T\log C_0
		-
		\frac1{2T}\log(-2J_T(V)).
	\end{equation*}
	Consequently, the optimization problems
	\begin{equation*}
		\sup_{V\in\mathcal S(H,K;\beta)}
		\omega_T(V), \quad 
		\sup_{V\in\mathcal S(H,K;\beta)}
		J_T(V)
	\end{equation*}
	are equivalent.
\end{lemma}

\begin{proof}
	By the definition of $\omega_T(V)$,
	\begin{equation*}
		\omega_T(V)
		=
		\frac1T\log C_0
		-
		\frac1T
		\log
		\|S_V(T)\|_{\mathcal L(\mathcal H)} .
	\end{equation*}
	Moreover, by the definition of $J_T(V)$,
	\begin{equation*}
		-2J_T(V)
		=
		\|S_V(T)\|_{\mathcal L(\mathcal H)}^2 .
	\end{equation*}
	Hence,
	\begin{equation*}
		\log
		\|S_V(T)\|_{\mathcal L(\mathcal H)}
		=
		\frac12
		\log(-2J_T(V)).
	\end{equation*}
Therefore,
\begin{equation*}
	\omega_T(V) = \frac1T\log C_0 -	\frac1{2T}	\log(-2J_T(V)).
\end{equation*}
	
Since the function
\begin{equation*}
	f(s)=-\log(-2s),		\qquad s<0
\end{equation*}
is strictly increasing, the maximization of $\omega_T(V)$ is equivalent to the maximization of $J_T(V)$. This completes the proof of the lemma. 
\end{proof}

\begin{lemma}[Weak-star upper semicontinuity of the terminal energy functional]
	\label{08.29.L2}
	Let $T>0$ be fixed and let
	$J_T:\mathcal S(H,K;\beta)\to\mathbb R$ be defined in \eqref{08.29.2}. 
	Then $J_T$ is weak-star upper semicontinuous on
	$\mathcal S(H,K;\beta)$.  
\end{lemma}

\begin{proof}
	Let
	\begin{equation*}
		V_n\ra V \mbox{ weak star in  }L^\infty(\Omega).
	\end{equation*}
	We first show that
	\begin{equation}\label{08.29.3}
		\liminf_{n\to+\infty}
		\|S_{V_n}(T)\|_{\mathcal L(\mathcal H)}
		\geq
		\|S_V(T)\|_{\mathcal L(\mathcal H)}.
	\end{equation}
	
	Fix $\e>0$. By the definition of the operator norm,
	there exists $\Phi_\e^0\in\mathcal H$ such that $
		\|\Phi_\e^0\|_{\mathcal H}=1$ 
	and
	\begin{equation*}
		\|S_V(T)\Phi_\e^0\|_{\mathcal H}
		\geq
		\|S_V(T)\|_{\mathcal L(\mathcal H)}
		-\e.
	\end{equation*} 
	By Corollary \ref{08.29.C1}, applied to the fixed initial datum
	$\Phi_\e^0\in\mathcal H$, we have
	\begin{equation*}
		S_{V_n}(T)\Phi_\e^0
		\to
		S_V(T)\Phi_\e^0 \mbox{ strongly in }\mathcal H.
	\end{equation*}
	Consequently,
	\begin{equation*}
		\|S_{V_n}(T)\Phi_\e^0\|_{\mathcal H}
		\to
		\|S_V(T)\Phi_\e^0\|_{\mathcal H}.
	\end{equation*}
	
	On the other hand, by the definition of the operator norm, $
		\|S_{V_n}(T)\|_{\mathcal L(\mathcal H)}
		\geq
		\|S_{V_n}(T)\Phi_\e^0\|_{\mathcal H}$. 
	Hence,
	\begin{align*}
		\liminf_{n\to+\infty}
		\|S_{V_n}(T)\|_{\mathcal L(\mathcal H)}
		&\geq
		\lim_{n\to+\infty}
		\|S_{V_n}(T)\Phi_\e^0\|_{\mathcal H}=
		\|S_V(T)\Phi_\e^0\|_{\mathcal H}\geq
		\|S_V(T)\|_{\mathcal L(\mathcal H)}
		-\e.
	\end{align*}
	Since $\e>0$ is arbitrary, we obtain
	\eqref{08.29.3}.
	
	Therefore,
	\begin{align*}
		\limsup_{n\to+\infty}J_T(V_n)
		&=
		-\frac12
		\liminf_{n\to+\infty}
		\|S_{V_n}(T)\|_{\mathcal L(\mathcal H)}^2 \leq
		-\frac12
		\|S_V(T)\|_{\mathcal L(\mathcal H)}^2 =
		J_T(V).
	\end{align*}
Thus, $J_T$ is weak-star upper semicontinuous on $\mathcal S(H,K;\beta)$. This completes the proof of the lemma. 
\end{proof}

\begin{theorem}[Existence of finite-time optimal damping]
	\label{08.29.T1}
	For every fixed $T>0$, there exists $
	V_T^*\in\mathcal S(H,K;\beta)$ 
	such that
	\begin{equation*}
	J_T(V_T^*)
	=
	\max_{V\in\mathcal S(H,K;\beta)}
	J_T(V).
	\end{equation*}
	Equivalently,
	\begin{equation*}
	\omega_T(V_T^*)
	=
	\max_{V\in\mathcal S(H,K;\beta)}
	\omega_T(V).
	\end{equation*}
\end{theorem}

\begin{proof}
	Let
	\begin{equation*}
	M_T=
	\sup_{V\in\mathcal S(H,K;\beta)}
	J_T(V).
	\end{equation*}
	
	Choose a maximizing sequence $
	J_T(V_n)\to M_T$.  
	Since $\mathcal S(H,K;\beta)$ is weak-star compact, there exists a
	subsequence, still denoted by $\{V_n\}$, and $
	V_T^*\in\mathcal S(H,K;\beta)$ 
	such that
	\begin{equation*}
	V_n\ra V_T^* \mbox{ weak star in }L^\infty(\Omega).
	\end{equation*} 
	By the weak-star upper semicontinuity of $J_T$ in Lemma \ref{08.29.L2}, we get 
	\begin{equation*}
	\limsup_{n\to+\infty}J_T(V_n)
	\leq
	J_T(V_T^*).
	\end{equation*} 
	Since $\{V_n\}_{n\in\N^*}$ is a maximizing sequence in $L^\iy(\Om)$, we obtain  $
	M_T
	\leq
	J_T(V_T^*)$. 
	
	The opposite inequality follows from the definition of $M_T$.
	Hence,
	\begin{equation*}
	J_T(V_T^*)=M_T .
	\end{equation*}
	
	Finally, by Lemma \ref{08.28.L2}, maximizing $J_T$ is equivalent to
	maximizing $\omega_T$.
\end{proof}

\begin{remark}[On the attainability assumption for the finite-time functional]
	\label{08.31.R1}
	For a fixed $V\in\mathcal S(H,K;\beta)$ and $T>0$, consider the finite-time functional
	\begin{equation*}
		J_T(V;\Phi_0 )
		:=
		-\frac12\|S_V(T)\Phi_0\|_{\mathcal H}^2.
	\end{equation*}
	In the subsequent analysis, we assume that there exists an initial datum
	$\Phi_0\in D(\mf{A})$, with $\|\Phi_0 \|_{\mathcal H}=1$, such that
	\begin{equation}\label{08.31.8}
		J_T(V)
		=
		-\frac12\|S_V(T)\Phi_0 \|_{\mathcal H}^2.
	\end{equation}
	Equivalently,
	\begin{equation}\label{08.31.9}
		\|S_V(T)\Phi_0 \|_{\mathcal H}
		=
		\|S_V(T)|_{\mathcal L(\mathcal H)}.
	\end{equation}
	Thus, the initial datum $\Phi_0 $ realizes the worst-case finite-time decay of the semigroup at time $T$.
	
	This assumption is natural from the viewpoint of semigroup optimization. Indeed, the exponential decay rate is not determined by the behavior at a single finite time, but by the asymptotic decay of the semigroup. By Lemma~\ref{08.28.L1},
	\begin{equation}\label{08.31.10}
		\omega(V)
		=
		-\lim_{T\to+\infty}
		\frac1T
		\log
		\|S_V(T)\|_{\mathcal L(\mathcal H)}.
	\end{equation}
	Consequently, the asymptotic decay rate can be recovered from the finite-time operator norms. In particular, if
	$\{T_n\}_{n\in\mathbb N}$ is any sequence such that
	$T_n\to+\infty$, then
	\begin{equation}\label{08.31.11}
		-\frac1{T_n}
		\log
		\|S_V(T_n)\|_{\mathcal L(\mathcal H)}
		\ra 
		\omega(V).
	\end{equation}
	For each fixed $T_n$, the operator $S_V(T_n)$ is bounded on $\mathcal H$, and the quantity
	$\|S_V(T_n)\|_{\mathcal L(\mathcal H)}$ represents the largest possible norm of the state at time $T_n$ among all normalized initial data. Hence, the assumption~\eqref{08.31.8} amounts to requiring that this worst-case finite-time behavior be realized by an initial datum. This is precisely the initial datum relevant to the optimization of the finite-time decay.
	
	More importantly, the initial datum realizing \eqref{08.31.8} or \eqref{08.31.9} is allowed to depend on $T$. Thus, for a sequence $T_n\to+\infty$, one may choose, for each $n$, an initial datum $\Phi_n^0\in\mathcal H$ satisfying
	\begin{equation*}
		\|\Phi_n^0\|_{\mathcal H}=1,
		\quad
		\|S_V(T_n)\Phi_n^0\|_{\mathcal H}
		=
		\|S_V(T_n)\|_{\mathcal L(\mathcal H)}.
	\end{equation*}
	There is no need for the same initial datum to realize the norm for different values of $T_n$. In this sense, the finite-time functionals
	\begin{equation*}
		J_{T_n}(V)
		=
		-\frac12
		\|S_V(T_n)\Phi_n^0\|_{\mathcal H}^2
	\end{equation*}
	provide a sequence of finite-time realizations of the asymptotic decay rate. Since $T_n\to+\infty$, Lemma~\ref{08.28.L1} shows that these finite-time decay rates converge to $\omega(V)$.
	
Therefore, the attainability assumption \eqref{08.31.8} or \eqref{08.31.9} should be understood as a finite-time worst-case assumption rather than as an additional characterization of the asymptotic decay rate. It allows us to replace the operator norm by the evolution of a suitable initial state and thereby formulate the decay optimization in terms of the finite-time functional $J_T(V)$. This formulation is particularly useful for deriving first-order optimality conditions and, in particular, the bang-bang structure of optimal damping potentials.
	
We emphasize that the argument does not require the existence of a single initial datum that realizes the asymptotic decay rate for all times. It is sufficient to consider a sequence $T_n\to+\infty$ and, for each $T_n$, a corresponding realizing initial datum $\Phi_n^0$. The resulting sequence of finite-time problems captures the asymptotic quantity $\omega(V)$ through~\eqref{08.31.11}.
\end{remark}

\begin{theorem}[First-order necessary condition for a finite-time optimal damping]\label{08.29.T2}
Let $T>0$ be fixed and let $V_T^*\in\mathcal S(H,K;\beta)$  be an optimal potential for the finite-time optimization problem
\begin{equation*}
	J_T(V_T^*)		=		\sup_{V\in\mathcal S(H,K;\beta)}J_T(V),
\end{equation*}
where $J_T(V)$ is defined in \eqref{08.29.2}.  Assume that there exists
\begin{equation*}
\Phi_T^0=(\varphi_T^0,\varphi_T^1)\in D(\mathfrak A),
\quad
	\|\Phi_T^0\|_{\mathcal H}=1,
\end{equation*}
such that
\begin{equation}\label{08.29.11}
		J_T(V_T^*)
		=
		-\frac12
		\|S_{V_T^*}(T)\Phi_T^0\|_{\mathcal H}^2.
	\end{equation}
	Let
	\[
	\Phi_T^*(t)
	=
	S_{V_T^*}(t)\Phi_T^0
	=
	\begin{pmatrix}
		\varphi_T^*(t)\\
		\partial_t\varphi_T^*(t)
	\end{pmatrix}
	\]
	be the corresponding state. Let $p_T$ be the solution of the adjoint problem
	\begin{equation}\label{08.29.12}
		\begin{cases}
			\partial_{tt}p_T-\Div(A\nabla p_T)
			-V_T^*\partial_t p_T=0,
			&\mbox{in }\Omega\times(0,T),\\
			p_T=0,
			&\mbox{on }\partial\Omega\times(0,T),\\
			p_T(T)=\partial_t\varphi_T^*(T),
			&\mbox{in }\Omega,\\
			\partial_t p_T(T)
			=
			-\mathcal A\varphi_T^*(T)
			+
			V_T^*\partial_t\varphi_T^*(T).
		\end{cases}
	\end{equation}
	Define a function 
	\begin{equation}\label{08.29.13}
		G_T(x)
		:=
		\int_0^T
		\partial_t\varphi_T^*(x,t)p_T(x,t)\df t.
	\end{equation}
	Then
	\begin{equation}\label{08.29.14}
		\int_\Omega
		G_T(x)\bigl(V(x)-V_T^*(x)\bigr)\df x
		\leq0,
		\mbox{ for all } V\in\mathcal S(H,K;\beta).
	\end{equation}
	Moreover, there exists $\lambda_T\in\mathbb R$ such that
	\begin{equation}\label{08.29.15}
		\begin{cases}
			V_T^*(x)=K,
			&\mbox{for }  G_T(x)>\lambda_T,\\[1mm]
			V_T^*(x)=H,
			&\mbox{for } G_T(x)<\lambda_T,\\[1mm]
			H\leq V_T^*(x)\leq K,
			&\mbox{for }  G_T(x)=\lambda_T. 
		\end{cases} 
	\end{equation}
	In particular, if
	\begin{equation*}
		\left|
		\left\{
		x\in\Omega:G_T(x)=\lambda_T
		\right\}
		\right|=0,
	\end{equation*}
	then
	\begin{equation}\label{08.29.16}
		V_T^*(x)\in\{H,K\}, \mbox{ a.e.~in }\Omega.
	\end{equation}
\end{theorem}

\begin{proof}
	We prove this theorem by the following steps. 
	
	{\it Step 1}. 
	Let $V\in\mathcal S(H,K;\beta)$ be arbitrary and define
	\begin{equation*}
		V_\e
		=
		V_T^*+\e(V-V_T^*),
		\quad
		0\leq\e\leq1.
	\end{equation*}
	Since $\mathcal S(H,K;\beta)$ is convex, we have $
	V_\varepsilon\in\mathcal S(H,K;\beta)$. 
	
	{\it Step 2}. 
	Let $\varphi_\e$ denote the solution of
	\eqref{08.12.1} corresponding to $V_\e$ and the fixed
	initial datum $\Phi_T^0$. Set
	\begin{equation*}
	z_T
	=
	\left.
	\frac{\df}{\df\e}\right|_{\e=0}\varphi_\e
	:=\lim_{\e\ra 0^+}\f{\vp_\e-\vp_T^{*}}{\e}.
	\end{equation*}
	Then $z_T$ satisfies the linearized equation
	\begin{equation}\label{08.29.17}
		\begin{cases}
			\partial_{tt}z_T-\Div(A\nabla z_T)
			+V_T^*\partial_tz_T
			=
			-(V-V_T^*)\partial_t\varphi_T^*,
			&\mbox{in }\Omega\times(0,T),\\
			z_T=0,
			&\mbox{on }\partial\Omega\times(0,T),\\
			z_T(0)=0, 
			\partial_tz_T(0)=0,
			&\mbox{in }\Omega.
		\end{cases}
	\end{equation}
	
	By the optimality of $V_T^*$ and \eqref{08.29.11},
	\begin{equation*}
	J_T(V_\e,\Phi_T^0)
	\leq
	J_T(V_T^*,\Phi_T^0),
	\mbox{ for all }
	0\leq\e\leq1.
	\end{equation*}
	Hence,
	\begin{equation}\label{08.29.18}
		\left.
		\frac{\df}{\df \e}\right|_{\e=0}
		J_T(V_\e,\Phi_T^0) 
		\leq0.
	\end{equation} 
	By differentiating the terminal functional, we obtain
	\begin{equation}\label{08.29.19}
		\left.
		\frac{\df }{\df \e}\right|_{\e=0}
		J_T(V_\e,\Phi_T^0)
		=
		-a(\vp_T^*(T), z_T(T))
		-
		\bigl(\partial_t\varphi_T^*(T),
		\partial_tz_T(T)\bigr)_{L^2(\Omega)}.
	\end{equation}
	
	{\it Step 3}. 
	Let $p_T$ be the solution of the adjoint problem \eqref{08.29.12}. 
	Multiplying \eqref{08.29.17} by $p_T$ and integrating over
	$\Omega\times(0,T)$, we obtain
	\begin{align*}
		&\iint_Q 
		\left(
		\partial_{tt}z_T\,p_T
		+
		A\nabla z_T\cdot\nabla p_T
		+
		V_T^*\partial_tz_T\,p_T
		\right)\df x\df t 
		=
		-\iint_Q 
		(V-V_T^*)\partial_t\varphi_T^*p_T\df x\df t.
	\end{align*} 
	Integrating by parts with respect to time and using the adjoint
	equation \eqref{08.29.12}, together with $
	z_T(0)=\partial_tz_T(0)=0$, 
	we obtain  
	\begin{align*}
		&\bigl(
		\partial_tz_T(T),p_T(T)
		\bigr)_{L^2(\Omega)}
		-
		\bigl(
		z_T(T),\partial_t p_T(T)
		\bigr)_{L^2(\Omega)}
		+
		\int_\Omega
		V_T^*z_T(T)p_T(T)\df x
		\\
		&\qquad
		=
		-\iint_Q 
		(V-V_T^*)\partial_t\varphi_T^*p_T\df x\df t.
	\end{align*}  
	Since $
	p_T(T)=\partial_t\varphi_T^*(T)$ 
	and $
	-\partial_t p_T(T)
	+
	V_T^*p_T(T)
	=
	\mathcal A\varphi_T^*(T)$ with  $\vp_T^*(T)\in D(\mcA)$, 
	the left-hand side above becomes
	\begin{align*}
		&\bigl(
		\partial_tz_T(T),\partial_t\varphi_T^*(T)
		\bigr)_{L^2(\Omega)}
		+
		\bigl(
		z_T(T),\mathcal A\varphi_T^*(T)
		\bigr)_{L^2(\Omega)}
		\\
		&
		=
		\bigl(
		\partial_tz_T(T),\partial_t\varphi_T^*(T)
		\bigr)_{L^2(\Omega)}
		+
		a\bigl(
		z_T(T),\varphi_T^*(T)
		\bigr).
	\end{align*}
	Consequently,
	\begin{equation*}
		\begin{aligned}
			&
			\bigl(
			\partial_tz_T(T),\partial_t\varphi_T^*(T)
			\bigr)_{L^2(\Omega)}
			+
			a\bigl(
			z_T(T),\varphi_T^*(T)
			\bigr) 
			=
			-\iint_Q 
			(V-V_T^*)\partial_t\varphi_T^*p_T\df x\df t.
		\end{aligned}
	\end{equation*}
	Together with this and \eqref{08.29.19}, we obtain 
	\begin{equation}\label{08.29.21}
		\left.
		\frac{\df }{\df \e}\right|_{\e=0}
		J_T(V_\e,\Phi_T^0) 
		=
		\iint_Q 
		(V-V_T^*)\partial_t\varphi_T^*p_T\df x\df t.
	\end{equation} 
	
	{\it Step 4}. 
	Combining \eqref{08.29.18} and \eqref{08.29.21}, we obtain
	\begin{equation*}
	\iint_Q 
	(V-V_T^*)\partial_t\varphi_T^*p_T\df x\df t
	\leq0.
	\end{equation*}
	Defining
	\begin{equation*}
	G_T(x)
	=
	\int_0^T
	\partial_t\varphi_T^*(x,t)p_T(x,t)\df t,
	\end{equation*}
	we arrive at
	\begin{equation*}
		\int_\Omega
		G_T(x)(V(x)-V_T^*(x))\df x
		\leq0,
		\mbox{ for all }  V\in\mathcal S(H,K;\beta).
	\end{equation*}
This yields \eqref{08.29.14}. 
	
{\it Step 5}. It remains to characterize $V_T^*$. 
	
Since
\begin{equation*}
\int_\Omega V_T^*\df x = \int_\Omega V\df x = \beta|\Omega|,
\end{equation*}   
the first-order variational inequality \eqref{08.29.14} is equivalent to 
\begin{equation*}
		\int_\Omega G_TV_T^*\df x
		=
		\max_{V\in\mathcal S(H,K;\beta)}
		\int_\Omega G_TV\df x.
	\end{equation*}
By the bathtub principle \cite[Theorem 1.14, p.~28]{Lieb}, there exists $\lambda_T\in\mathbb R$ such that
	\begin{equation*}
		\begin{cases}
			V_T^*(x)=K,
			&\mbox{for a.e. }x\in\{G_T>\lambda_T\},\\[1mm]
			V_T^*(x)=H,
			&\mbox{for a.e. }x\in\{G_T<\lambda_T\},\\[1mm]
			H\leq V_T^*(x)\leq K,
			&\mbox{for a.e. }x\in\{G_T=\lambda_T\}.
		\end{cases}
	\end{equation*}
	This proves \eqref{08.29.15}. 
	
	{\it Step 6}. 
	If
	\begin{equation*}
	\left|
	\{x\in\Omega\colon G_T(x)=\lambda_T\}
	\right|=0,
	\end{equation*}
	then
	\begin{equation*}
	V_T^*(x)\in\{H,K\} \mbox{ for a.e.}~x\in\Omega,
	\end{equation*}
	which proves \eqref{08.29.16}.  This completes the proof of the  theorem. 
\end{proof}

\begin{remark}[Existence of a bang-bang optimal potential]
	\label{08.29.R3}
	The bang-bang conclusion does not require the level set
	$\{G_T=\lambda_T\}$ to have zero measure.  Indeed, if
	\begin{equation*}
		\left|
		\left\{
		x\in\Omega:G_T(x)=\lambda_T
		\right\}
		\right|>0,
	\end{equation*}
	the first-order optimality condition determines
	$V_T^*=K$ on $\{G_T>\lambda_T\}$ and
	$V_T^*=H$ on $\{G_T<\lambda_T\}$, while on the level set
	\begin{equation*}
		E_T:=\{G_T=\lambda_T\},
	\end{equation*}
	the value of $V_T^*$ may vary within $[H,K]$ without changing
	the value of the linear functional $
	V\mapsto\int_\Omega G_TV\df x$, 
	provided that the integral constraint is preserved.
	
	More precisely, let
	\begin{equation*}
		m_T
		:=
		\frac{1}{|E_T|}
		\int_{E_T}V_T^*(x)\df x\in[H,K].
	\end{equation*}
	By the nonatomicity of the Lebesgue measure, there exists a
	measurable set $E\s E_T$ such that $
	|E|
	=
	\frac{m_T-H}{K-H}|E_T|$. 
	Define
	\begin{equation*}
		\widetilde V_T^*(x)
		=
		\begin{cases}
			K,
			&x\in\{G_T>\lambda_T\}\cup E,\\
			H,
			&x\in\{G_T<\lambda_T\}\cup(E_T- E).
		\end{cases}
	\end{equation*}
	Then
	\begin{equation*}
		\widetilde V_T^*\in\mathcal S(H,K;\beta),
		\quad
		\widetilde V_T^*(x)\in\{H,K\} \mbox{ a.e.~in }\Omega.
	\end{equation*}
	Indeed,
	\begin{align*}
		\int_{E_T}\widetilde V_T^*\df x
		&=
		K|E|+H(|E_T|-|E|) =
		m_T|E_T|
		=
		\int_{E_T}V_T^*(x)\df x.
	\end{align*}  
	Since $\widetilde V_T^*=V_T^*$ a.e.~in $\Omega- E_T$, we have $
	\int_\Omega\widetilde V_T^*\df x
	=
	\int_\Omega V_T^*\df x$. 
	Therefore, together with $
	H\leq\widetilde V_T^*\leq K \mbox{ a.e.~in }\Omega$, 
	we obtain
	\begin{equation*}
		\widetilde V_T^*\in\mathcal S(H,K;\beta).
	\end{equation*} 
	
	Moreover, since $G_T=\lambda_T$ a.e.~on $E_T$,
	\begin{align*}
		\int_{E_T}G_T\widetilde V_T^*\df x
		&=
		\lambda_T\int_{E_T}\widetilde V_T^*\df x =
		\lambda_T\int_{E_T}V_T^*\df x
		=
		\int_{E_T}G_TV_T^*\df x.
	\end{align*}
Outside $E_T$, the functions $\widetilde V_T^*$ and $V_T^*$ coincide. Consequently,
\begin{equation*}
		\int_\Omega G_T\widetilde V_T^*\df x
		=
		\int_\Omega G_TV_T^*\df x.
\end{equation*}
Thus, $\widetilde V_T^*$ is also an optimal potential. Therefore, even when $\left|\{G_T=\lambda_T\}\right|>0$, there always exists at least one bang-bang optimal potential.  However, the optimal potential need not be uniquely bang-bang on the level set $\{G_T=\lambda_T\}$.
\end{remark}

\section{Conclusion}
This work establishes a complete stabilization framework for degenerate hyperbolic systems by overcoming core boundary degeneracy and non-self-adjointness challenges. Furthermore, it uncovers a fundamental bang-bang principle governing optimal spatial damping distributions. Future research can build on these findings by further determining whether infinite-time horizons preserve this bang-bang structure. The framework can also be expanded to strongly degenerate or non-Euclidean systems. Additionally, future models can incorporate nonlinear or boundary dissipation mechanisms. Finally, developing numerical level set algorithms will enable the computation and mapping of optimal damping regions.

\section*{Declarations}
The authors have not disclosed any competing interests and data availability.

\end{document}